\documentclass[11pt]{article}

\usepackage[english]{babel}
\usepackage[utf8]{inputenc}
\usepackage[T1]{fontenc}
\usepackage[toc,page]{appendix}
\usepackage{graphicx}
\usepackage{geometry}
\usepackage{bbm}
\usepackage{dsfont}
\usepackage{alltt}
\usepackage{verbatimbox}
\usepackage{amsthm,amsfonts}
\usepackage{amsmath}
\usepackage{amssymb}
\usepackage{mathabx}
\usepackage{accents}
\usepackage{titlesec}
\usepackage{mathtools}
\usepackage{empheq}

\mathtoolsset{showonlyrefs=true}

\newcommand{\real}{\mathbb{R}}

\newcommand{\T}{\mathbb{T}}

\newcommand{\x}{\xi}

\newcommand{\jap}[1]{\left\langle #1 \right\rangle}

\numberwithin{equation}{section}
\newcommand{\be}{\begin{equation}}
\newcommand{\ee}{\end{equation}}

\theoremstyle{plain}
\newtheorem{thm}{Theorem}
\newtheorem*{thm*}{Theorem}
\newtheorem{prop}{Proposition}
\newtheorem{cor}{Corollary}
\newtheorem{assumpt}{Assumption}
\newtheorem{lem}{Lemma}

\theoremstyle{definition}

\theoremstyle{remark}
\newtheorem{nb}{Remark}

\makeatletter
\def\blfootnote{\xdef\@thefnmark{}\@footnotetext}
\makeatother

\title{Nonlinear smoothing for dispersive PDE: a unified approach}

\date{}%\nodate
\author{Simão Correia and Jorge Drumond Silva}

\begin{document}
\maketitle

\begin{abstract}
	In several cases of nonlinear dispersive PDEs, the difference between the nonlinear and linear evolutions with the same initial data, i.e. the integral term in Duhamel's formula, exhibits improved regularity. This property is usually called \textit{nonlinear smoothing}. We employ the method of infinite iterations of normal form reductions to obtain a very general theorem yielding existence of nonlinear smoothing for dispersive PDEs, contingent only on establishing two particular bounds. We then apply this theorem to show that the nonlinear smoothing property holds, depending on the regularity of the initial data, for five classical dispersive equations: in $\mathbb{R}$, the cubic nonlinear Schrödinger, the Korteweg-de Vries, the modified Korteweg-de Vries and the derivative Schrödinger equations; in $\mathbb{R}^2$, the modified Zakharov-Kuznetsov equation. For the aforementioned one-dimensional equations, this unifying methodology matches or improves the existing nonlinear smoothing results, while enlarging the ranges 
	of Sobolev regularities where the property holds.
\end{abstract}

\section{Introduction}
Many examples of nonlinear dispersive PDEs have been known to exhibit a property, known as \textit{nonlinear smoothing}, whereby the difference between the nonlinear and linear evolutions for the same initial data, given by the integral term in Duhamel's formula, has higher regularity than the initial data. 

The
phenomenon seems to first have been proved and used by Bona and Saut \cite{BS1}, in the
context of showing the existence of
dispersive blow-up for the generalized KdV equations, which was a question that had previously been raised in \cite{BBM}. Indeed, \textit{dispersive blow-up} is an intimately related phenomenon in which singularities occur in the flow of
nonlinear dispersive equations due exclusively to their linear component, while the nonlinear part of the evolution is shown to have a  more regular behavior.
Therefore, proofs of dispersive blow-up type results necessarily always hinge on a central ingredient of
guaranteeing some form of improved regularity for the nonlinear terms, when compared to the linear evolution. 
In this framework, the study started by Bona and Saut for gKdV has since been  followed by several other works for different cases of dispersive equations (see e.g. \cite{BPSS,BS2,FLPS}). 

Currently, however, the usage of the nonlinear smoothing designation has become more specifically
associated to results where the increase of regularity is measured in terms of Sobolev space derivatives, and that is not necessarily
always the case when seeking to prove dispersive singularity formation.
In this stricter sense, the work of Linares and Scialom  in \cite{linaresscialom} was the first to objectively address the issue of gain of regularity of the integral term in
Duhamel's formula for
solutions of the modified KdV equation, in terms of Sobolev space regularity. Later, analogous types of results were proved by Bourgain (see \cite{bourgain}) for the cubic, defocusing, nonlinear Schr\"odinger equation in $\real^2$, and then for the general $L^2$-critical nonlinear Schr\"odinger equation in $\real^n$ by Keraani and Vargas (\cite{keraani}). More recently, M. Erdo\u{g}an, N. Tzirakis and their
collaborators have extensively studied this property for several classes of equations and settings, in one spatial dimension (see \cite{tzirakis7, tzirakis1, tzirakis2} and references therein) by using a combination of normal form reduction together with Bourgain's Fourier restriction norm method ($X^{s,b}$ spaces).

In spite of this extensive research on nonlinear smoothing and related properties for solutions of dispersive equations, no unifying theory seems to exist though. In this paper, we exploit the method of the infinite iteration of normal form reductions, developed in \cite{ko} and \cite{koy} to study unconditional well-posedness, to obtain a very general theorem yielding existence of nonlinear smoothing for dispersive PDEs, dependent only on
guaranteeing that two particular bounds hold.

Let us consider thus a generic dispersive PDE with a nonlinear term of monomial type
\begin{equation}\label{eq:geral}
u_t + i L(D)u = N(u), \quad u\big|_{t=0}=u_0\in H^s(\real^d),
\end{equation}
where $L(D)$, $D=\partial/i$, is a linear differential (or pseudo-differential) operator  given by a real Fourier symbol $L(\xi)$. If $(G(t))_{t\in\real}$ denotes the associated linear group, given by 
$G(t)u_0=e^{-itL(D)}u_0=\mathcal{F}^{-1}\left( e^{-itL(\xi)}\widehat{u_0}(\xi)\right) $, then the solution
to this initial value problem is given in integral form  by the Duhamel formula:
\begin{equation}\label{fullduhamel}
u(t)=G(t)u_0 +\int_0^t G(t-t')N(u)(t')dt'.
\end{equation}

We assume the existence of a nice local well-posedness theory: 
\begin{assumpt}
Given $u_0\in H^{s_0}(\real^d)$, there exist $T=T(\|u_0\|_{H^{s_0}})$ and a solution $u\in C([0,T],H^{s_0}(\real^d))$ of \eqref{eq:geral}, depending continuously on the initial data. Moreover, if $u_0\in H^s(\real^d)$, $s\ge {s_0}$, then $u\in C([0,T], H^s(\real^d))$.
\end{assumpt}

We are interested in the nonlinear smoothing property of \eqref{eq:geral}, which can be stated as follows: 
for some $s\geq s_0$, given initial data $u_0\in H^s(\real^d)$, the difference between the full nonlinear evolution
of the initial value problem \eqref{eq:geral}, represented by \eqref{fullduhamel}, and the corresponding linear evolution, for the same initial data, satisfies
$$
\|u(t)-G(t)u_0\|_{H^{s+\epsilon}}=\left\|\int_0^t G(t-t')N(u)(t')dt'\right\|_{H^{s+\epsilon}}\le C(t,\|u_0\|_{H^s}).
$$
This means that the nonlinear integral term in the Duhamel formula \eqref{fullduhamel} for $u$ is smoother than the linear evolution by $\epsilon
=\epsilon(s)$ 
Sobolev derivatives, at the level of $H^s$ regularity of the flow of \eqref{eq:geral}. 
\vskip10pt
A very convenient quantity in the study of nonlinear smoothing properties  of dispersive equations (just as it is for scattering or other properties where one wants to focus on the behavior of the nonlinear factor in \eqref{fullduhamel} by performing integration by parts or the stationary phase method) 
 is the so called profile, or the interaction representation, of the solution,  $G(-t)u(t)$. By canceling out the linear evolution from the full solution, its Fourier representation exhibits the
advantage of isolating all
the oscillations in the nonlinear term. Denoting by $\tilde{u}(t,\xi)=\mathcal{F}(G(-t)u(t))(\xi)$ its Fourier 
transform in the space variable, it is easy to see that, for a general equation as in \eqref{eq:geral}, one gets
\begin{equation}\label{eq:tildeu}
\tilde{u}(t,\x)=\tilde{u}(0,\x) + \int_0^t\int_{\x_1+\dots+\x_k=\x} e^{is\Phi(\Xi)}m(\Xi)\tilde{u}(s, \x_1)\dots \tilde{u}(s, \x_k)d\x_1\dots d\x_{k-1} ds
\end{equation}
where $\Xi=(\xi,\xi_1,\dots, \xi_k)$, $m$ is some multiplicative symbol related to possible derivatives occurring in the
nonlinearity $N(u)$ and $k$ is the order of the nonlinearity (for simplicity, we have omitted any possible complex conjugation). 
Then, as we will see below, the method pursued will lead to the following two crucial multilinear operators 
\begin{equation}
\mathcal{F}\left[T_\sigma(u_1,\dots,u_k)\right](\x)=\int_{\x_1+\dots + \x_k=\x} \frac{1}{\jap{\Phi(\Xi)}^\sigma}m(\Xi)\hat{u}_1(\x_1)\dots \hat{u}_k(\x_k) d\x_1\dots d\x_{k-1},
\end{equation}
$$
\mathcal{F}[T^{\alpha, M}(u_1,\dots, u_k)](\xi)=\int_{\substack{\x_1+\dots + \x_k=\x\\|\Phi(\Xi)-\alpha|<M}}m(\Xi)\hat{u}_1(\xi_1)\dots \hat{u}_k(\xi_k)d\x_1\dots d\x_k,
$$
such that, for any general dispersive equation, if the following two bounds can be established,
\begin{equation}\label{eq:basicaepsilon}\tag{Bound$_{\sigma,\epsilon}$}
\|T_\sigma(u_1,\dots, u_k)\|_{H^{s+\epsilon}}\lesssim \prod_{j=1}^{k} \|u_j\|_{H^s},
\end{equation}
\begin{equation}\label{eq:frequencyrestricted}\tag{Bound$^{\alpha, M}$}
\|T^{\alpha,M}(u_1,\dots,u_k)\|_{H^s}\lesssim \jap{\alpha}^{0^+}\jap{M}^{1/2^+}\prod_{j=1}^k\|u_j\|_{H^s},\quad \alpha, M\in \real
\end{equation}
then nonlinear smoothing holds, with a gain of regularity of $\epsilon$ Sobolev derivatives. This is the central, generic nonlinear smoothing result of this paper.
\begin{thm}[Nonlinear Smoothing]\label{thm:nonlinearsmooth}
	Suppose that \eqref{eq:basicaepsilon} and \eqref{eq:frequencyrestricted} hold for some $\epsilon>0$, $\sigma\in (0,1)$. Then the flow on $H^s(\real^d)$ generated by \eqref{eq:geral} possesses a nonlinear smoothing effect of order $\epsilon$.
\end{thm}

\begin{nb}
	As a consequence of our method, one may easily derive the Lipschitz estimate
\begin{equation}\label{eq:LipBound}
	\|u(t)-v(t)-G(t)(u_0-v_0)\|_{H^{s+\epsilon}}\le C(t, \|u_0\|_{H^s}, \|v_0\|_{H^s})\|u_0-v_0\|_{H^s}.
\end{equation}
	Moreover, if \eqref{eq:basicaepsilon} and \eqref{eq:frequencyrestricted} hold for some $s$ and $\epsilon$, it is easy to check that they remain valid for any $s'>s$, with the same gain of regularity $\epsilon$.
\end{nb}

\begin{nb}
	It is worth pointing out that Theorem \ref{thm:nonlinearsmooth} remains valid with a slightly more
	general form of estimate \eqref{eq:frequencyrestricted}:
	\begin{equation}\label{eq:frequencyrestrictedgen}
	\|T^{\alpha,M}(u_1,\dots,u_k)\|_{H^s}\lesssim \sup\{\jap{\alpha}^{\gamma},\jap{M}^{\gamma}\}\jap{M}^{\beta}\prod_{j=1}^k\|u_j\|_{H^s},\quad \alpha, M\in \real,
	\end{equation}
	with constraints connected to \eqref{eq:basicaepsilon}:
	$$
	\gamma+\beta<1,\quad \sigma+\gamma<1.
	$$
	The proof of the result in this more general form requires a slight modification of  the estimates involved in the infinite normal form reduction. However, no vital novelty is involved as the main ideas and concepts of the proof remain the same, and the small increase in generality comes at the cost of making the presentation significantly more burdensome. Therefore, for the sake of clarity, we will restrict our arguments to the simpler case
	\eqref{eq:frequencyrestricted} and briefly explain
	the necessary adaptations when using \eqref{eq:frequencyrestrictedgen}  in Remark \ref{gammabeta}.
	Furthermore, regarding the examples that we present below, this inequality shall only  be applied to prove nonlinear smoothing for the derivative nonlinear Schr\"odinger equation.

\end{nb}

%
%\vskip10pt
%
%As we shall observe, the above theorem is applicable for several dispersive equations in $H^s$, for any $s$ in the range for local well-posedness. These examples raise the possibility that, for any given dispersive equation, as long as a strong local well-posedness theory is available, some degree of nonlinear smoothing is verified. In this direction, using a simple interpolation argument, we provide a partial answer:
%
%\begin{prop}[Nonlinear smoothing vs. local well-posedness]\label{prop:interpol}
%	Suppose that \eqref{eq:LipBound} holds for some $s\in \real$ and let $s_1<s$ be such that \eqref{eq:geral} is locally well-posed in $H^{s_1}(\real^d)$. 
%%	Moreover, assume the following persistence of regularity estimate:
%%	\begin{equation}
%%	\|u\|_{C([0,T],H^s)} \lesssim \|u_0\|_{H^{s_1}}^{k-1}\|u_0\|_{H^s}, \quad s_1<s.
%%	\end{equation}
%	Given $\theta \in (0,1)$, set
%	$$
%	s'= s_1 + \theta (s-s_1).
%	$$
%	Then the flow on $H^{s'}(\real^d)$ generated by \eqref{eq:geral} possesses a nonlinear smoothing effect of order $\theta\epsilon$.
%\end{prop}
%

\subsection{Applications}

As a result of Theorem \ref{thm:nonlinearsmooth}, one may derive several nonlinear smoothing results for dispersive equations. In this work, we have chosen five classical examples:

\begin{enumerate}
	\item The modified Korteweg-de Vries equation on $\real$:
	\begin{equation}\label{mkdv}\tag{mKdV}
	u_t + u_{xxx} \pm (u^3)_x = 0, \quad u_0\in H^s(\real).
	\end{equation}
	Following \cite{KPV5, KPV4}, the Cauchy problem is well-posed for $s\ge 1/4$ and ill-posed for $s<1/4$. As mentioned above, the first and still best known result regarding nonlinear smoothing for the (mKdV) equation on $\real$ was obtained by Linares and Scialom in \cite{linaresscialom}, who showed that, for any $s\ge 3/4$, one has a gain of one derivative. The proof of this result relies strongly on the algebraic relations of the \eqref{mkdv}, which allow for a perfect combination of smoothing and maximal function estimates. On the torus $\T$, nonlinear smoothing for all regularities starting at the lowest local well-posedness $s>1/2$ was established in \cite{tzirakis1},
	proving that for the (mKdV) equation in this setting there is a gain of $\min\{2s-1,1\}$ derivatives in
	the nonlinear part of the evolution. This also yields a gain of a full derivative for
	initial data in $H^s(\T)$, when $s>1$. We shall prove
	\begin{cor}\label{teo:mkdv}
		For any $s>1/4$, the \eqref{mkdv} flow on $H^s(\real)$ presents a nonlinear smoothing of any order $\epsilon < \min\{ (4s-1)/2,1\}$.
	\end{cor}
	Observe that, for any $s>1/4$, one has some gain of regularity, reaching almost one derivative exactly when $s=3/4$. We argue that this is a strong indicator of the strength of our method, as it covers the best known result without relying on the fine algebraic structure of \eqref{mkdv}.

	\item The Korteweg-de Vries equation:
	\begin{equation}\label{kdv}\tag{KdV}
	u_t + u_{xxx} \pm (u^2)_x = 0, \quad u_0\in H^s(\real).
	\end{equation}
	In this context, local well-posedness has been shown for $s\ge -3/4$ using Bourgain spaces (see \cite{KPV6, NTT}). It is also worth referring \cite{killipvisan}, where local existence (in a weaker sense) was proven in $H^{-1}(\real)$. Using the associated vector field $\Gamma=x-t\partial_{xx}^2$ and $X^{s,b}$ spaces, 
	it was shown in \cite{nahasponce} that, if $u_0\in H^s(\real)\cap L^2(|x|^{s/2}dx)$, then $u(t)\in L^2(|x|^{s/2}dx)$ for any $t>0$ (see also Appendix A). Concerning nonlinear smoothing, for any $s>7/6$, if $u_0\in H^{s}(\real)\cap L^2(|x|^{s/2}dx)$, then one has a gain of $1/6$ derivatives (\cite{linaresponcesmith}, \cite{linarespalacios}).  For periodic boundary conditions, one has nonlinear smoothing with a gain of $\min\{2s+1,1\}$, for all $s>-1/2$ (see \cite{tzirakis1}),
	which is
	the lowest regularity for LWP of the \eqref{kdv} equation on $\T$ (\cite{KPV6}). Here, we greatly improve the existing theory on $\real$:
	\begin{cor}\label{teo:kdv}
		For any $s>0$, the \eqref{kdv} flow on $ H^s(\real)\cap L^2(|x|^{s/2}dx)$ presents a nonlinear smoothing of any order $\epsilon < \min\{s,1\}$.
	\end{cor}
	
	\item The cubic nonlinear Schr\"odinger equation in dimension one:
	\begin{equation}\label{nls}\tag{NLS}
	iu_t + u_{xx} \pm |u|^2u = 0,\quad u_0\in H^s(\real).
	\end{equation}
	The local well-posedness theory for $s\ge 0$ is a consequence of the standard Strichartz estimates. This range is optimal (\cite{CCT, KPV4}), in the sense that uniform continuous dependence on the initial data fails for any $s<0$. Regarding the known nonlinear smoothing results, we refer to \cite[Proposition 4.1]{BPSS}, where it has been shown that, for $s>1/4$, one has a gain of $1/2$ derivatives and, for $s>1/2$, one achieves a gain of a full derivative. 
	%More recently, Erdo\u{g}an, G\"urel and Tzirakis proved that, for $s>1/2$, the gain of the nonlinear factor in the Duhamel formula for \eqref{nls} on $\real$ is almost a full derivative (see \cite{tzirakis7}). 
	In \cite{tzirakis7}, the authors showed that,
	on $\T$, there is an increase in regularity of order $\min\{2s,1\}$, for $s>1/4$, which agrees with the previous 
	values for the whole line. We prove exactly the same range of gains, on $\real$, for $s>0$.
	\begin{cor}\label{teo:nls}
		For any $s>0$, the \eqref{nls} flow on $H^s(\real)$ presents a nonlinear smoothing of any order $\epsilon < \min\{2s,1\}$.
	\end{cor}
\item The modified Zakharov-Kuznetzov equation:
\begin{equation}\label{mzk}\tag{mZK}
u_t + \partial_x\Delta u = \partial_x(u^3),\quad u_0\in H^s(\real^2).
\end{equation}
The best known range for the local well-posedness theory is $s>1/4$ (\cite{ribaudvento}) and it remains unknown whether it is optimal, as the scaling critical space index is $s_c=0$. 
%In a recent paper (\cite{FLPS}), where
%dispersive blow-up and nonlinear smoothing properties are studied for the Zakharov-Kuznetzov and the modified
%Zakharov-Kuznetzov equations, the authors prove that there is a nonlinear smoothing gain of one full derivative for initial data
%in $H^2(\real^2)$. 
A key observation made in \cite{grunrockherr} is the fact that, through a simple change of coordinates, one may symmetrize the \eqref{mzk}:
$$
u_t + u_{xxx} + u_{yyy} = (\partial_x + \partial_y)(u^3).
$$
The separation between the two spatial directions allows the application of more refined dispersive estimates. Using this symmetrization, we prove
% 
% \begin{cor}\label{teo:mZK}
% 	For any $s>1/4$, the \eqref{mzk} flow on $H^s(\real^2)$ presents a nonlinear smoothing of any order $\epsilon < \min\{ (4s-1)/7, 1\}$.
% \end{cor}
 \begin{cor}\label{teo:mZK}
	For any $s>3/2$, the \eqref{mzk} flow on $H^s(\real^2)$ presents a nonlinear smoothing of any order $\epsilon < \min\{ 2s-3, 1\}$.
\end{cor}
%To our knowledge, it was previously unknown if the \eqref{mzk} possessed any sort of nonlinear smoothing property. 
%\begin{nb}
%	As far as we have managed, we can only apply Theorem \ref{thm:nonlinearsmooth} for $s>3/2$ and show a gain of regularity of order $\epsilon<\min\{2s-3,1\}$. This is the only example where we fail to cover the whole local well-posedness range using Theorem \ref{thm:nonlinearsmooth}. The full range claimed in Corollary \ref{teo:mZK} is achieved through Proposition \ref{prop:interpol}.
%\end{nb}

\begin{nb}
Here, the minimal $s$ is $3/2$, which is quite far from the known local well-posedness theory. In this example, our aim was to show that this technique is not restricted to the one-dimensional case and that even very crude estimates yield some degree of nonlinear smoothing. We believe that this result can be improved through a more careful analysis of the bounds \eqref{eq:basicaepsilon} and \eqref{eq:frequencyrestricted}. In fact, in a parallel paper \cite{FLPS}, it is shown that a stronger smoothing property holds, based on the usual Kato smoothing estimate. In a future work, we will address bidimensional problems in detail and present more refined results.
\end{nb}

\item The derivative nonlinear Schr\"odinger equation
\begin{equation}\label{dnls}\tag{dNLS}
iu_t + u_{xx} + i\partial_x(|u|^2u)=0.
\end{equation}
As it was proven in \cite{Takaoka}, the initial value problem is well-posed in $H^s(\real)$ for $s\ge 1/2$, using the Fourier restriction method. Recently, the normal form reduction has been applied in \cite{mosincatyoon} to prove unconditional uniqueness for any $s\ge 1/2$, that is, uniqueness without the auxiliary Bourgain space. In either case, it is necessary to apply a (bijective) gauge transformation
$$
w(t,x)= \exp\left(-i\int_{-\infty}^x |u(t,y)|^2 dy\right)u(t,x)
$$
which connects \eqref{dnls} with
\begin{equation}\label{gdnls}\tag{dNLS*}
iw_t + w_{xx} = - i |w|^2 \partial_x \bar{w} - \frac{1}{2}|w|^4w.
\end{equation}
%Erdo\u{g}an, G\"urel and Tzirakis have studied nonlinear smoothing properties for \eqref{dnls} on the half line (\cite{tzirakis6}) and showed that, for $1/2<s<5/2$, with $s \neq 3/2$ there is a gain of regularity of of order $\min\{2s-1,1/4, 5/2-s\}$.
 For \eqref{gdnls}, we prove the following.
\begin{cor}\label{cor:dnls}
	For any $s>1/2$, the \eqref{gdnls} flow on $H^s(\real)$ presents a nonlinear smoothing of any order $\epsilon<\min \{ 2s-1, 1/2 \}$.
\end{cor}

\end{enumerate}

\begin{nb}
Excluding the \eqref{mzk} example, we prove that some degree of nonlinear smoothing holds whenever $s$ is larger than the limit for the local well-posedness theory. We conjecture that this property holds for general dispersive equations. More precisely, if some degree of nonlinear smoothing of order $\epsilon$ is valid on $H^{s_1}$, for some large $s_1$, and if the initial value problem is locally well-posed on $H^{s_0}$, $s_0<s_1$, then one should be able to prove nonlinear smoothing on $H^s$, for any $s_0<s<s_1$. Moreover, the gain of regularity should be at least $\theta\epsilon$, where $s=\theta s_1 + (1-\theta)s_0$.
\end{nb}

%\section{The frequency-restricted estimates}
%In this section, we prove \eqref{eq:frequencyrestricted} for each equation. For the cubic (NLS) and the (mKdV), such an estimate may be found in [Oh-Yoon]. Here, we present their proofs due to some major simplifications in the arguments.

\section{Description of the method and proof of the main results}

In this section, we explain how the infinite normal form reduction method introduced in \cite{guo}, with precursors
in \cite{babin} and \cite{ko}, can be used to prove nonlinear smoothing. To keep the exposition concise, we present briefly the general idea of the INFR as in \cite[Section 3]{koy} and indicate the specific modifications that allow the improvement in regularity. For more details regarding the specifics of the INFR, we advise the reader to consult \cite{ko, koy, mosincatyoon}.

\subsection{The algorithm of the INFR}

As already mentioned, the equation for the Fourier transform in space of the profile $\tilde{u}(t)=\mathcal{F}(G(-t)u(t))$ reads
\begin{align}
\tilde{u}(t,\x)&=\tilde{u}(0,\x) + \int_0^t\int_{\x_1+\dots+\x_k=\x} e^{is\Phi(\Xi)}m(\Xi)\tilde{u}(s, \x_1)\dots \tilde{u}(s, \x_k)d\x_1\dots\nonumber d\x_{k-1} ds\\& = \tilde{u}(0,\x) + \int_0^t \mathcal{N}^{(1)}(\tilde{u}(s))ds .\label{eq:tildeu1}
\end{align}
The common procedure at this point, in the normal form reduction method first used by Shatah \cite{shatah}, is to integrate by parts in time once as in the (non)stationary phase method, to exploit the oscillating factor
in the nonlinear term. This has the advantage of producing increased regularity
due to the powers of the frequency in the phase $\Phi(\Xi)$ that the method generates in the denominator, but only as long as the frequencies remain
away from the problematic resonance $\Phi(\Xi)=0$. Of course the difficulty of this method rests precisely on somehow controlling these problematic resonances. Erdo\u{g}an and Tzirakis, first inspired by
a normal form reduction method for the KdV equation in \cite{babin}, base their general methodology for proving 
nonlinear smoothing on controlling the resonance through the Fourier restriction norm method of Bourgain, or
$X^{s,b}$ spaces. However, they do mention (\cite{tzirakis7}) that this technique is mostly useful for equations on
the torus $\T$, as it becomes harder to apply for the real line $\real$. The breakthrough idea, in \cite{guo}, was not to stop the integration by parts after a finite number of times, but instead to proceed infinitely many times.

For a fixed $N>1$ (which will be determined later), one splits the frequency domain into the near-resonant and nonresonant regions, depending on whether $|\Phi|$ is smaller or greater than $N$. We use the subscripts $1$ for the near-resonant term and $2$ for the nonresonant one:
$$
\tilde{u}(t,\x)=\tilde{u}(0,\x) + \int_0^t \mathcal{N}_1^{(1)}(\tilde{u}(s)) +\mathcal{N}_2^{(1)}(\tilde{u}(s)) ds.
$$
Using \eqref{eq:basicaepsilon}, 

\begin{equation}\label{eq:primeirotermo}
\left\|\mathcal{F}^{-1}\left(\mathcal{N}_1^{(1)}(\tilde{u}(t))\right)\right\|_{H^{s+\epsilon}}\lesssim N^\sigma\|T_\sigma(\tilde{u},\dots,\tilde{u})(t)\|_{H^{s+\epsilon}}\lesssim N^\sigma\|u(t)\|_{H^s}^{k}.
\end{equation}

For $\mathcal{N}_2^{(1)}$, one integrates by parts in time, using the relation
$$
e^{is\Phi}=\partial_s\left(\frac{1}{i\Phi}e^{is\Phi}\right)
$$
so that
$$
\int_0^t \mathcal{N}_2^{(1)}(\tilde{u}(s)) ds = \left[\mathcal{N}_0^{(2)}(\tilde{u}(s)) \right]_{s=0}^{s=t} + \int_0^t \mathcal{N}^{(2)}(\tilde{u}(s)) ds.
$$
Due to the factor $1/\Phi$ and the restriction on the frequency domain, the boundary terms are easily bounded:
$$
\left\|\mathcal{F}^{-1}\left(\mathcal{N}_0^{(2)}(\tilde{u}(s))\right)\right\|_{H^{s+\epsilon}}\lesssim \frac{1}{N^{1-\sigma}}\|T_\sigma(\tilde{u},\dots,\tilde{u})(s)\|_{H^{s+\epsilon}}\lesssim N^{1-\sigma}\|u(s)\|_{H^s}^{k}.
$$
For the remainder $\mathcal{N}^{(2)}$, one uses \eqref{eq:tildeu1} to replace $\tilde{u}_t$ and obtain a new oscillatory integral which is of order $2(k-1)+1$ in $\tilde{u}$. 

So far, we have rewritten \eqref{eq:tildeu1} as
$$
\tilde{u}(t,\xi) = \tilde{u}(0,\xi) + \mbox{controllable terms} + \int_0^t \mathcal{N}^{(2)}(\tilde{u}(s))ds.
$$

This concludes the first step in the INFR. One may now apply a recursive algorithm to expand the remainder integral $\mathcal{N}^{(2)}$. Fix $0<\delta<1$ and set
$$
\beta_j=((k-1)(j+1)+1)^k,\quad j\in \mathbb{N}.
$$
\vskip10pt
\underline{At the $J$-th step,}
\vskip10pt
\textit{Step 1.} Split the frequency domain into
\begin{itemize}
	\item Near-resonant: $\mbox{|Phase|}<\beta_J\max\{\mbox{|Phase at step 1|}^{1-\delta}, \mbox{|Phase at previous step|}^{1-\delta}\}$;
	\item Nonresonant: $\mbox{|Phase|}>\beta_J\max\{\mbox{|Phase at step 1|}^{1-\delta}, \mbox{|Phase at previous step|}^{1-\delta}\}$;
\end{itemize}
and write $$\int_0^t\mathcal{N}^{(J+1)}(\tilde{u}(s))ds=\int_0^t\mathcal{N}_1^{(J+1)}(\tilde{u}(s))+\mathcal{N}_2^{(J+1)}(\tilde{u}(s))ds.
$$
\textit{Step 2.} Integrate by parts in time the nonresonant term:
$$
\int_0^t \mathcal{N}_2^{(J+1)}(\tilde{u}(s))ds = \left[\mathcal{N}_0^{(J+2)}(\tilde{u}(s))  \right]_{s=0}^{s=t} + \int_0^t \mathcal{N}^{(J+2)}(\tilde{u}(s))ds.
$$
\textit{Step 3.} Use \eqref{eq:tildeu1} to replace all instances of $\tilde{u}_t$ in order to obtain an oscillatory integral of order $(k-1)(J+2)-1$ in $\tilde{u}$.

\textit{Step 4.} Repeat the algorithm for $\mathcal{N}^{(J+2)}(\tilde{u})$.

\vskip10pt

After an infinite amount of steps, one formally obtains the normal form equation
\begin{equation}\label{eq:NFE}\tag{NFE}
\tilde{u}(t,\xi)=\tilde{u}(0,\xi) + \sum_{j\ge 2} \left[\mathcal{N}_0^{(j)}(\tilde{u}(s))\right]_{s=0}^{s=t} + \int_0^t \sum_{j\ge 1} \mathcal{N}_1^{(j)}(\tilde{u}(s))ds.
\end{equation}

\subsection{Bounds for the \eqref{eq:NFE}}

If one does not desire to gain any regularity, \eqref{eq:frequencyrestricted} is enough to control all the terms in the normal form equation:
\begin{lem}[\cite{koy}, Lemmas 3.15 and 3.16]\label{lem:koy}
 Assume \eqref{eq:frequencyrestricted}. Then, for any $J\ge 1$,
 
\begin{equation}
 \left\|\mathcal{F}^{-1}\left(\mathcal{N}_0^{(J+1)}(\tilde{u}(t))\right)\right\|_{H^s}\lesssim N^{-\frac{1}{2} - (J-1)\frac{1-\delta}{2} + 0^+}\|u(t)\|_{H^s}^{(k-1)J+1}
\end{equation}

\begin{equation}\label{eq:estlema316}
 \left\|\mathcal{F}^{-1}\left(\mathcal{N}_1^{(J+1)}(\tilde{u}(t))\right)\right\|_{H^s}\lesssim N^{-\frac{1}{2} - (J-2)\frac{1-\delta}{2}+ 0^+}\|u(t)\|_{H^s}^{(k-1)(J+1)+1}
\end{equation}
 
\end{lem} 

Let us explain briefly the proof of the last estimate (the first is easier). First of all, notice that any of these terms is a byproduct of $\mathcal{N}_2^{(1)}$, meaning that $|\Phi(\Xi)|>N$. Then one may write
\begin{align*}
|\mathcal{N}_1^{(J+1)}(\tilde{u}(t))| &\lesssim \int_{|\Phi|>N} \frac{1}{|\Phi|}v_1(\xi_1)\dots v_k(\xi_k) d\xi_1\dots d\xi_k \\&\lesssim \sum_{\substack{M>N\\M\ \mbox{\footnotesize dyadic}}} \int_{|\Phi|\sim M} \frac{1}{|\Phi|}v_1(\xi_1)\dots v_k(\xi_k) d\xi_1\dots d\xi_k.
\end{align*}
where $v_1,\dots, v_k$ depend on $\tilde{u}$. Applying \eqref{eq:frequencyrestricted},
\begin{align}
\left\|\mathcal{F}^{-1}\left(\mathcal{N}_1^{(J+1)}(\tilde{u}(t))\right)\right\|_{H^s}&\lesssim \sum_{\substack{M>N\\M\ \mbox{\footnotesize dyadic}}} M^{-1} \|T^{0,M}(v_1,\dots, v_k)\|_{H^s}\nonumber\\& \lesssim \sum_{\substack{M>N\\M\ \mbox{\footnotesize dyadic}}} M^{-1/2+0^+} \prod_{l=1}^k\|v_l\|_{H^s}.\label{eq:provalema316}
\end{align}
Due to the restrictions in the frequencies, using once again \eqref{eq:frequencyrestricted}, one may show (see \cite[Proof of Lemma 3.16]{koy}) that

\begin{equation}\label{eq:estvl}
\prod_{l=1}^k \|v_l\|_{H^s} \lesssim M^{-(J-2)\frac{1-\delta}{2}+0^+} \|u(t)\|_{H^s}^{(k-1)(J+1)+1}
\end{equation}

and so
\begin{align*}
\left\|\mathcal{F}^{-1}\left(\mathcal{N}_1^{(J+1)}(\tilde{u}(t))\right)\right\|_{H^s}\lesssim \sum_{\substack{M>N\\M\ \mbox{\footnotesize dyadic}}} M^{-1/2-(J-2)\frac{1-\delta}{2}+0^+} \|u(t)\|_{H^s}^{(k-1)(J+1)+1}.
\end{align*}
Since $0<\delta<1$, the sum is finite for any $J\ge 1$ and \eqref{eq:estlema316} follows.

Now we argue how \eqref{eq:basicaepsilon} can improve the above argument in order to obtain a gain of regularity. Instead of applying \eqref{eq:frequencyrestricted} in \eqref{eq:provalema316}, we use \eqref{eq:basicaepsilon}:
\begin{align*}
\left\|\mathcal{F}^{-1}\left(\mathcal{N}_1^{(J+1)}(\tilde{u}(t))\right)\right\|_{H^{s+\epsilon}}&\lesssim \sum_{\substack{M>N\\M\ \mbox{\footnotesize dyadic}}} M^{-(1-\sigma)} \|T_\sigma(v_1,\dots, v_k)\|_{H^{s+\epsilon}}\\& \lesssim \sum_{\substack{M>N\\M\ \mbox{\footnotesize dyadic}}} M^{-(1-\sigma)} \prod_{l=1}^k\|v_l\|_{H^s}.
\end{align*}
It follows from \eqref{eq:estvl} that
\begin{equation}\label{eq:melhorcomsigma}
\left\|\mathcal{F}^{-1}\left(\mathcal{N}_1^{(J+1)}(\tilde{u}(t))\right)\right\|_{H^{s+\epsilon}}\lesssim \sum_{\substack{M>N\\M\ \mbox{\footnotesize dyadic}}} M^{-(1-\sigma)-(J-2)\frac{1-\delta}{2}+0^+} \|u(t)\|_{H^s}^{(k-1)(J+1)+1}.
\end{equation}
If $J=1$ and $\delta$ is small, the above sum is divergent. However, choosing $\delta=\sigma$, the sum is once again finite for any $J\ge 1$. Therefore, we may state
\begin{lem}\label{lem:boundsgeral}
	Assume \eqref{eq:basicaepsilon} and \eqref{eq:frequencyrestricted}. Take $\delta=\sigma$ in the INFR. Then, for any $J\ge 1$,
	\begin{equation}
	\left\|\mathcal{F}^{-1}\left(\mathcal{N}_0^{(J+1)}(\tilde{u}(t))\right)\right\|_{H^{s+\epsilon}}\lesssim N^{-(1-\sigma) - (J-1)\frac{1-\sigma}{2} + 0^+}\|u(t)\|_{H^s}^{(k-1)J+1}
	\end{equation}
	and
	\begin{equation}\label{eq:estlema316novo}
	\left\|\mathcal{F}^{-1}\left(\mathcal{N}_1^{(J+1)}(\tilde{u}(t))\right)\right\|_{H^{s+\epsilon}}\lesssim N^{-(1-\sigma) - (J-2)\frac{1-\sigma}{2}+ 0^+}\|u(t)\|_{H^s}^{(k-1)(J+1)+1}
	\end{equation}
\end{lem}

\begin{nb}\label{gammabeta} Under the more general bound \eqref{eq:frequencyrestrictedgen}, estimate \eqref{eq:estvl} becomes
	\begin{align*}
	\prod_{l=1}^k\|v_l\|_{H^s}\lesssim M^{\gamma-(J-2)(1-\delta)(1-\eta)}\|u(t)\|_{H^s}^{(k-1)(J+1)+1},\quad \eta=\gamma+\beta.
	\end{align*}
	and, analogously to \eqref{eq:melhorcomsigma},
	\begin{equation}
	\left\|\mathcal{F}^{-1}\left(\mathcal{N}_1^{(J+1)}(\tilde{u}(t))\right)\right\|_{H^{s+\epsilon}}\lesssim \sum_{\substack{M>N\\M\ \mbox{\footnotesize dyadic}}} M^{-(1-\sigma-\gamma)-(J-2)(1-\delta)(1-\eta) } \|u(t)\|_{H^s}^{(k-1)(J+1)+1}
	\end{equation}
	Recalling that $\sigma+\gamma,\eta <1$, if $\delta$ is sufficiently close to 1, the sum is convergent. These estimates imply analogous bounds to those of Lemma \ref{lem:boundsgeral}.
\end{nb}

\subsection{Validity of the \eqref{eq:NFE} for smooth data and proof of Theorem 1}
In order for the \eqref{eq:NFE} to have any usefulness, one must show that any solution of \eqref{eq:geral} must also solve \eqref{eq:NFE}. This justification requires that
\begin{itemize}
	\item one may apply the INFR algorithm for rough solutions;
	\item the remainder at step $J$ tends to $0$ in some weaker norm, as $J\to\infty$.
\end{itemize}
At low regularity, the proof of the above properties can be quite involved and it is actually the hardest part in the normal form method. Luckily, in order to prove nonlinear smoothing, we only need to show the validity of the \eqref{eq:NFE} for some large level of regularity, where these properties become trivial.

Take $u_0\in H^m$, $m\gg s,d/2$. Then $u\in C([0,T], H^m)$ and it follows directly from the equation that $u_t \in C([0,T], H^s)$ and
$$
\|u_t(t)\|_{H^s}\lesssim \|u(t)\|_{H^m}^k, \quad t\in [0,T].
$$

Therefore one may apply the INFR algorithm (which includes integration by parts in time and an exchange between time derivatives and frequency integration). It remains to see that the remainder $\mathcal{N}^{(J+1)}(\tilde{u}(t))$ goes to zero in some weaker norm. However, observe that, \textit{before replacing }$\tilde{u}_t$\textit{ using }\eqref{eq:tildeu1}, the remainder is essentially the boundary term $\mathcal{N}_0^{(J+1)}$, with $\tilde{u}_t$ instead of one $\tilde{u}$. Therefore, applying the same arguments as in Lemma \ref{lem:koy},
\begin{lem}
	Assume \eqref{eq:frequencyrestricted}. Suppose that $u_0\in H^m$, $m\gg s, d/2$. Then the remainder $\mathcal{N}^{(J+1)}(\tilde{u}(t))$ satisfies
	$$
	 \left\|\mathcal{F}^{-1}\left(\mathcal{N}^{(J+1)}(\tilde{u}(t))\right)\right\|_{H^s}\lesssim N^{-\frac{1}{2} - (J-1)\frac{1-\delta}{2} + 0^+}\|u(t)\|_{H^m}^{(k-1)(J+1)+1}
	$$
	Consequently, $\tilde{u}$ satisfies the \eqref{eq:NFE} in the distributional sense.
\end{lem}

\begin{proof}[Proof of Theorem \ref{thm:nonlinearsmooth}]
	Take $u_0\in H^m(\real^d)$, $m\gg s, d/2$. Then $\tilde{u}$ satisfies \eqref{eq:NFE} and, due to the local well-posedness theory on $H^s(\real^d)$, the solution is defined on $[0,T]$, $T=T(\|u_0\|_{H^s})$ and satisfies
	$$
	\|u(t)\|_{H^s}\le 2\|u_0\|_{H^s}=:M.
	$$
	Take
	$$
	N=(2M^{k-1})^{\frac{2}{1-\sigma}}.
	$$
	Applying \eqref{eq:primeirotermo} and the estimates of Lemma \ref{lem:boundsgeral}, for any $0<t<T$,
	\begin{align*}
	\|\mathcal{F}^{-1}\left(\tilde{u}(t)-\tilde{u}(0)\right)\|_{H^{s+\epsilon}} &\lesssim \sum_{j\ge 2} \left\|\mathcal{F}^{-1}\left(\mathcal{N}_0^{(j)}(\tilde{u}(t)) \right)\right\|_{H^{s+\epsilon}} + \sum_{j\ge 2} \left\|\mathcal{F}^{-1}\left(\mathcal{N}_0^{(j)}(\tilde{u}(0)) \right)\right\|_{H^{s+\epsilon}}\\&\qquad + \sum_{j\ge 1} \int_0^t \left\|\mathcal{F}^{-1}\left(\mathcal{N}_1^{(j)}(\tilde{u}(s)) \right)\right\|_{H^{s+\epsilon}}ds\\&\lesssim \sum_{j\ge 2} N^{-(1-\sigma) - (j-2)\frac{1-\sigma}{2} + 0^+}M^{(k-1)j+1} + TN^{1-\sigma}\\&\qquad +T\sum_{j\ge 2}N^{-(1-\sigma) - (j-3)\frac{1-\sigma}{2}+ 0^+}M^{(k-1)(j+1)+1}\\&\lesssim M^{1+0^+} + TM^{2k-2}+ TM^{2k-1+0^+} =C(\|u_0\|_{H^s}), \quad \forall u_0\in H^m(\real^d).
	\end{align*}
	If $v_0\in H^s(\real^d)$, one approximates by $(v_0^n)_{n\in\mathbb{N}}$, $v_0^n\in H^m(\real^d)$, and, by continuous dependence,
\begin{align*}
	\|v(t)-G(t)v_0\|_{H^{s+\epsilon}}&=\|\mathcal{F}^{-1}\left(\tilde{v}(t)-\tilde{v}(0)\right)\|_{H^{s+\epsilon}}\le \liminf \|\mathcal{F}^{-1}\left(\tilde{v}^n(t)-\tilde{v}^n(0)\right)\|_{H^{s+\epsilon}}\\& \le \liminf C(\|v_0^n\|_{H^s}) = C(\|v_0\|_{H^s}).
\end{align*}
	The proof is finished.
\end{proof}

%
%\subsection{Proof of Proposition \ref{prop:interpol}}
%
%Fix $R>0$. Define $T=T(R)$ such that, for any $u_0\in H^{s_1}$, the corresponding solution $u$ of \eqref{eq:geral} is defined on $[0, T]$ and satisfies
%$$
%\|u\|_{C([0,T], H^{s_1}(\real^d))}\lesssim \|u_0\|_{H^{s_1}}.
%$$
%Define
%$$
%A_0, B_0 = H^{s_1}(\real^d), A_1 = H^s(\real^d), B_1=H^{s+\epsilon}(\real^d).
%$$
%For any $0<t<T$, consider the mapping
%$$
%Ku_0 = \left\{\begin{array}{ll}
%(R-\|u_0\|_{H^{s_1}})(u(t)- G(t)u_0), & \|u_0\|_{H^{s_1}}<R\\
%0, & \|u_0\|_{H^{s_1}}\ge R
%\end{array}\right..
%$$
%It follows that $K:A_j\to B_j$, $j=0,1$, and
%$$
%\|Ku_0\|_{B_0}\lesssim \|u_0\|_{A_0}, \quad u_0\in A_0.
%$$
%Moreover, using \eqref{eq:LipBound}, one may easily check that
%$$
%\|Ku_0-Kv_0\|_{B_1}\lesssim C(\|u_0\|_{A_1}, \|v_0\|_{A_1})\|u_0-v_0\|_{A_1}, \quad u_0,v_0 \in A_1.
%$$
%The result now follows from \cite[Theorem 2]{tartar}.

\section{Application of Theorem 1 to some classical examples}

Through Theorem 1, the proof of nonlinear smoothing is reduced to the study of the operators $T^{\alpha, M}$ and $T_\sigma$. In order to show the necessary bounds, we argue by duality: indeed, \eqref{eq:basicaepsilon} is equivalent to
\begin{equation}
I_{\sigma,\epsilon}:=\left|\int_{\x_1+\dots + \x_k=\x} \frac{\jap{\xi}^{s+\epsilon}}{\jap{\xi_1}^{s}\dots \jap{\xi_k}^{s}\jap{\Phi(\Xi)}^\sigma}m(\Xi)v_1(\x_1)\dots v_k(\x_k)v(\xi)\label{eq:integralse} d\x_1\dots d\x_{k-1}d\xi\right|\lesssim \prod_{j=1}^{k+1}\|v_k\|_{L^2},
\end{equation}
with $v_{k+1}:=v$, while \eqref{eq:frequencyrestricted} follows from
\begin{align}
I^{\alpha,M}&:=\left|\int_{\substack{\x_1+\dots + \x_k=\x\\|\Phi(\Xi)-\alpha|<M}}m(\Xi)\frac{\jap{\xi}^{s+\epsilon}}{\jap{\xi_1}^{s}\dots \jap{\xi_k}^{s}}v_1(\x_1)\dots v_k(\x_k)v(\xi) d\x_1\dots d\x_{k-1}d\xi\right| \nonumber\\&\lesssim \jap{\alpha}^{0^+}\jap{M}^{1/2^+}\prod_{j=1}^{k+1}\|v_k\|_{L^2}.\label{eq:integralam}
\end{align}
Moreover, when symmetry allows, we consider $|\xi_1|\ge \dots \ge|\xi_k|$.

\vskip10pt

The method we used for proving \eqref{eq:frequencyrestricted}, in the examples presented, follows the principle that, if one could integrate in $\Phi$ (through a suitable change of variables), then the estimate would follow from a Cauchy-Schwarz argument applied to the characteristic function of the set $\{|\Phi(\Xi)-\alpha|<M\}$. The only case where such a procedure is impossible is when the frequencies are close to a stationary point of $\Phi$. Assuming that these points form a finite set, one may essentially replace $\Phi$ by its value at each stationary point.
\vskip10pt
For \eqref{eq:basicaepsilon}, the main philosophy is that, when at least two frequencies are comparable to $\xi$, one may automatically compensate for the extra $\epsilon$ and the phase does not need to intervene (at least for large $s$). If not, then the phase can be simplified in order to involve only $\xi$ and $\xi_1$ and it can used to compensate the $\epsilon$-factor.
\vskip10pt
One useful estimate for integrals of type \eqref{eq:integralam} or \eqref{eq:integralse} can be achieved using the Cauchy-Schwarz inequality: if $\mathcal{M}$ is a symbol on $\xi, \xi_1, \dots, \xi_k$,
\begin{align}
&\int_{\substack{\x_1+\dots + \x_k=\x}} \mathcal{M}(\Xi)v_1(\xi_1)\dots v_k(\xi_k)v(\xi)d\x_1\dots d\xi_{k-1}d\xi\nonumber \\&\le \int \left(\int_{\substack{\x_1+\dots + \x_k=\x}} (\mathcal{M}(\Xi))^2 d\x_1\dots d\xi_{k-1}\right)^{\frac{1}{2}}\left(\int_{\substack{\x_1+\dots + \x_k=\x}} v_1^2(\xi_1)\dots v_k^2(\xi_k)d\x_1\dots d\xi_{k-1}\right)^{\frac{1}{2}}v(\xi)d\xi\nonumber  \\&\le \sup_{\xi } \left(\int_{\substack{\x_1+\dots + \x_k=\x}} (\mathcal{M}(\Xi))^2 d\x_1\dots d\xi_{k-1}\right)^{\frac{1}{2}} \prod_{j=1}^{k+1}\|v_j\|_{L^2}.\label{eq:sup}
\end{align}
Moreover, by symmetry, one can derive the same inequality but with $\xi$ and $\xi_j$ interchanged, for any $j=1,\dots, k-1$.

\subsection{Nonlinear smoothing for the (mKdV)}
For the (mKdV), we have
$$
\Phi(\Xi)=3(\xi-\xi_1)(\xi-\xi_2)(\xi-\xi_3),\quad m(\Xi)=\xi.
$$
Without loss of generality, we assume that $|\xi_1|\ge |\xi_2|\ge |\xi_3|$. Hence $|\x_1|\gtrsim |\xi|$.

\begin{lem}
	For any $s>1/4$, the estimate \eqref{eq:frequencyrestricted} holds for the \eqref{mkdv}.
\end{lem}

\begin{proof}
	In this context, 
	$$
	\frac{\partial \Phi}{\partial \xi_1}=3(\xi-\xi_2)(\xi_3-\xi_1), \frac{\partial \Phi}{\partial \xi_2}=3(\xi-\xi_1)(\xi_3-\xi_2).
	$$
	There are four stationary points:
	$$
	(\xi_1,\xi_2)=(\xi,\xi),\ (\xi,-\xi),\ (-\xi,\xi),\ \left(\frac{\xi}{3},\frac{\xi}{3}\right).
	$$
	\textbf{Case 1.} For any stationary point $P$, $|(\xi_1,\xi_2)-P|\gtrsim |\xi|$. Then either
	$$
	\left|\frac{\partial \Phi}{\partial \xi_1}\right|\gtrsim |\xi|^2\quad \mbox{or}\quad \left|\frac{\partial \Phi}{\partial \xi_2}\right|\gtrsim |\xi|^2.
	$$
	Assuming the first possibility, using \eqref{eq:sup}, one has
	\begin{align*}
	I^{\alpha,M} = &\left|\int_{\substack{\x_1+\dots + \x_k=\x}} \mathbbm{1}_{|\Phi-\alpha|<M} \frac{\jap{\xi}^{s+1}}{\jap{\xi_1}^s\jap{\xi_2}^s\jap{\xi_3}^s}v_1v_2v_3vd\x_1d\x_3d\x \right| \\ \lesssim  &	\sup_\xi \left(\int\mathbbm{1}_{|\Phi-\alpha|<M} \frac{\jap{\xi}^2}{\jap{\xi_3}^{4s}} d\x_1 d\x_3 \right)^{1/2}\prod_{j=1}^4\|v_j\|_2 \\ \lesssim & \sup_\xi \left(\int\mathbbm{1}_{|\Phi-\alpha|<M} \frac{\jap{\xi}^2}{\jap{\xi_3}^{4s}} \left|\frac{\partial \Phi}{\partial \xi_1}\right|^{-1}d\Phi d\x_3 \right)^{1/2}\prod_{j=1}^4\|v_j\|_2 \\ \lesssim & \left(\int\mathbbm{1}_{|\Phi-\alpha|<M} \frac{1}{\jap{\xi_3}^{4s}} d\Phi d\x_3 \right)^{1/2}\prod_{j=1}^4\|v_j\|_2 \lesssim M^{1/2}\prod_{j=1}^4\|v_j\|_2.
	\end{align*}

	\noindent\textbf{Case 2.} $|(\xi_1,\xi_2)-P|\ll |\xi|$, for some stationary point $P$.
	
	\noindent\textit{Case a)}. $P=(\xi/3,\xi/3)$. Then
	$$
	\Phi=3\xi^3Q,\quad Q\sim (2/3)^3,\quad  \frac{\jap{\xi}^{s+1}}{\jap{\xi_1}^s\jap{\xi_2}^s\jap{\xi_3}^s} \lesssim \jap{\xi}^{1/2}.
	$$
	Therefore
	\begin{align*}
	I^{\alpha,M}  \lesssim &\sup_{\xi_2}\left(\int \xi\mathbbm{1}_{|\Phi-\alpha|<M} \ d\x_1 d\x \right)^{1/2}\prod_{j=1}^4\|v_j\|_2 \\ \lesssim & \sup_{\xi_2}\left(\int \xi^2\mathbbm{1}_{|8\xi^3/9-\alpha|<M}  d\x \right)^{1/2}\prod_{j=1}^4\|v_j\|_2 \lesssim M^{1/2}\prod_{j=1}^4\|v_j\|_2.
	\end{align*}
	
	\noindent\textit{Case b)}. $P\neq (\xi/3,\xi/3)$. W.l.o.g., $P=(\xi,\xi)$. In this case,
	$$
	\Phi=\xi^3Q,\quad Q\sim (p_1-1)(p_2-1),\quad p_j=\frac{\xi_j}{\xi},\ j=1,2,
	$$
	and
	$$
	\frac{\jap{\xi}^{s+1}}{\jap{\xi_1}^s\jap{\xi_2}^s\jap{\xi_3}^s} \lesssim \jap{\xi}^{1/2^-}.
	$$
	Thus, for $\delta>0$ sufficiently small,
	\begin{align*}
	I^{\alpha,M} = &\left|\int_{\substack{\x_1+\dots + \x_k=\x}} \mathbbm{1}_{|\Phi-\alpha|<M} \frac{\jap{\xi}^{s+1}}{\jap{\xi_1}^s\jap{\xi_2}^s\jap{\xi_3}^s}v_1v_2v_3vd\x_1d\x_3d\x \right|  \\ \lesssim &\sup_{\xi }\left(\int \xi^{1^-}\mathbbm{1}_{|\Phi-\alpha|<M} \ d\x_1 d\x_2 \right)^{1/2}\prod_{j=1}^4\|v_j\|_2 \\ \lesssim & \sup_{\xi}\left(\int_{|p_1-1|,|p_2-1|<1} \xi^{3^-}\mathbbm{1}_{|\Phi-\alpha|<M}  dp_1dp_2 \right)^{1/2}\prod_{j=1}^4\|v_j\|_2 \\ \lesssim& \sup_{\xi}\left(\int_{|p_2-1|<1} \xi^{3^-} \left( \int \mathbbm{1}_{|\xi^3(p_1-1)(p_2-1)-\alpha|<M}  dp_1\right)^{\frac{1}{1+\delta}}dp_2 \right)^{1/2}\prod_{j=1}^4\|v_j\|_2  \\ \lesssim& \sup_{\xi}\left(\int_{|p_2-1|<1} \xi^{3^-} \frac{M^{\frac{1}{1+\delta}}}{\xi^{\frac{3}{1+\delta}}|p_2-1|^{\frac{1}{1+\delta}}}dp_2 \right)^{1/2}\prod_{j=1}^4\|v_j\|_2 \lesssim \jap{M}^{1/2}\prod_{j=1}^4\|v_j\|_2.
	\end{align*}
\end{proof}

\begin{lem}
	In the conditions of Corollary \ref{teo:mkdv}, estimate \eqref{eq:basicaepsilon} holds.
\end{lem}

\begin{proof}
The estimate is trivial if $|\xi|<1$:
\begin{align*}
I_{\sigma,\epsilon}&\le \left\|\left(\frac{v_1}{\jap{\cdot}^s}\right)\ast \left(\frac{v_2}{\jap{\cdot}^s}\right) \ast \left(\frac{v_3}{\jap{\cdot}^s}\right)\right\|_{L^2(-1,1)}\|v\|_{L^2} \le \left\|\left(\frac{v_1}{\jap{\cdot}^s}\right)\ast \left(\frac{v_2}{\jap{\cdot}^s}\right) \ast \left(\frac{v_3}{\jap{\cdot}^s}\right)\right\|_{L^\infty(-1,1) }\|v\|_{L^2}\\& \le \|v\|_{L^2}\prod_{j=1}^3 \left\|\frac{v_j}{\jap{\cdot}^s}\right\|_{L^{3/2}}\le \prod_{j=1}^4 \|v_j\|_{L^2}.
\end{align*}
Hence, from now on, we consider $|\xi|>1$.

\textbf{Case 1}. If $|\xi_2|\ll |\xi|$, then
$$
\Phi \sim \xi^2(\xi-\xi_1).
$$
If $|\xi-\xi_1|<1$,
\begin{align*}
I_{\sigma, \epsilon} &\lesssim \int m(\Xi)v_1(\xi_1)v(\xi)\left(\int v_2(\xi_2)v_3(\xi_3)d\xi_2\right)d\xi_1d\xi \\&\lesssim \left(\int_{|\xi-\xi_1|<1}\frac{\jap{\xi}^{1+\epsilon}}{\jap{\Phi}^\sigma} v_1(\xi_1)v(\xi)d\xi_1d\xi \right)\left\| \int v_2(\xi_2)v_3(\xi_3)d\xi_2\right\|_{L^\infty_{\xi,\xi_1}}\\&\lesssim\left(\int_{|\xi-\xi_1|<1}\frac{\jap{\xi}^{1+\epsilon}}{|\xi^2(\xi-\xi_1)|^\sigma} v_1(\xi_1)v(\xi)d\xi_1d\xi\right) \|v_2\|_2\|v_3\|_2\\&\lesssim \left(\int_{|\xi-\xi_1|<1}\frac{\jap{\xi}^{1+\epsilon-2\sigma}}{|\xi-\xi_1|^\sigma} v_1(\xi_1)v(\xi)d\xi_1d\xi \right)\|v_2\|_2\|v_3\|_2\\&\lesssim \left\|\int_{|\xi-\xi_1|<1}\frac{\jap{\xi}^{1+\epsilon-2\sigma}}{|\cdot-\xi_1|^\sigma} v_1(\xi_1)d\xi_1 \right\|_{L^2_\xi}\|v_2\|_2\|v_3\|_2\|v\|_{L^2}.
\end{align*}
If
$
2\sigma>1+\epsilon,\ \sigma<1,
$
the estimate follows from H\"older's inequality. On the other hand, if $|\xi-\xi_1|>1$, then $\jap{\Phi}^\sigma \gtrsim \jap{\xi}^{2\sigma}\jap{\xi-\xi_1}$. Then, since $s>1/4$ and $2\sigma>1+\epsilon$,
\begin{align*}
I_{\sigma, \epsilon} &\lesssim \int \frac{\xi^{1+\epsilon-2\sigma}}{\jap{\xi-\xi_1}^\sigma}v(\xi)v_1(\xi_1)v_2(\xi_2)\frac{v_3(\xi_3)}{\jap{\xi_3}^{2s}}d\xi_1d\xi_2d\xi_3\\&\lesssim \left\|\int \frac{1}{\jap{\xi-\xi_1}^\sigma}v(\xi)v_1(\xi_1)v_2(\xi_2)d\xi d\xi_1 \right\|_{L^\infty_{\xi_3}} \int \frac{v_3(\xi_3)}{\jap{\xi_3}^{2s}}d\xi_3\\&\lesssim \|v_1\|_{L^2}\|v_3\|_{L^2}\left\|\int \frac{1}{\jap{\xi-\xi_1}^\sigma}v(\xi)v_2(\xi-\xi_1-\xi_3)d\xi \right\|_{L^\infty_{\xi_3}L^2_{\xi_1}}\\&\lesssim \|v_1\|_{L^2}\|v_3\|_{L^2}\|v\|_{L^2}\left\| \frac{1}{\jap{\cdot}^\sigma}v_2(\cdot-\xi_3)\right\|_{L^\infty_{\xi_3} L^1_{\xi}}\lesssim \prod_{j=1}^{4}\|v_j\|_{L^2},
\end{align*}
where we used Young's inequality in the penultimate step.

\noindent\textbf{Case 2.} If $|\xi_2|\gtrsim |\xi|$: define $\tilde{\x}_j=\xi-\xi_j$. Choose $\alpha$ such that
$$
1+\epsilon-2s-\sigma(1-\alpha)\le 0,\quad -2s+1+\epsilon-\alpha\le 0.
$$
This is possible when $s>(1+2\epsilon)/4$.

\noindent\textit{Case a).} There are two $\tilde{\xi}_j$ which are small: w.l.o.g., $|\tilde{\xi}_1|<1$, $|\tilde{\xi}_2|<1$. Then $|\tilde{\xi}_3|, |\xi_3|\sim |\xi|$. If $|\tilde{\xi}_1|>|\xi|^{-\alpha}$,
\begin{align*}
I_{\sigma, \epsilon} &\lesssim \int \frac{\xi^{1+\epsilon}}{\jap{\xi_2}^s\jap{\xi_3}^s\jap{\Phi}^\sigma}vv_1v_2v_3  \lesssim \int \frac{\xi^{1+\epsilon-2s}}{|\tilde{\xi}_1\tilde{\xi}_2\tilde{\xi}_3|^\sigma}vv_2v_3v_1 \\&\lesssim \int \frac{\xi^{1+\epsilon-2s-\sigma(1-\alpha)}}{|\tilde{\xi}_2|^\sigma}vv_2v_3v_1 \lesssim \int \frac{1}{|\tilde{\xi}_2|^\sigma}vv_2d\xi d\xi_2 \left\|\int v_1v_3d\xi_1\right\|_{L^\infty_{\xi,\xi_2}}\lesssim \left\| |\cdot|^{-\sigma}\ast v\right\|_{L^2}\prod_{j=1}^3 \|v_j\|_{L^2}.
\end{align*}
The case $|\tilde{\xi}_2|>|\xi|^{-\alpha}$ is analogous. If $|\tilde{\xi}_1|,|\tilde{\xi}_2|<|\xi|^{-\alpha}$, we do not make use of $\Phi$. Instead, write
$$
\chi(\eta,\xi)=\mathbbm{1}_{|\eta-\xi|<|\xi|^{-\alpha}},\quad  K(\xi,\xi_3)=\int \xi^{-2s+1+\epsilon}\chi(\xi_1,\xi)\chi(\xi_2,\xi)v_1(\xi_1)v_2(\xi_2)d\xi_1.
$$
Then
$$
I_{\sigma, \epsilon} \lesssim \int K(\xi,\xi_3)v(\xi)v_3(\xi_3)d\xi d\xi_3.
$$
We now apply Schur's test to $K$: fixed $\xi$,
$$
\int K(\xi,\xi_3)d\xi_3 = \xi^{-2s+1+\epsilon}\int \chi(\xi_1,\xi)\chi(\xi_2,\xi)v_1v_2d\xi_1d\xi_2 \lesssim \xi^{-2s+1+\epsilon}\|\chi(\cdot, \xi)\|_{L^2}^2\|v_1\|_{L^2}\|v_2\|_{L^2}.
$$
Fixed $\xi_3$, recalling that $|\xi + \xi_3|= |\tilde{\xi}_1 + \tilde{\xi}_2|<2 |\xi|^{-\alpha}$,

\begin{align*}
\int K(\xi,\xi_3)d\xi &= \int \xi^{-2s+1+\epsilon} \chi(\xi_1,\xi)\chi(\xi_2,\xi)v_1v_2d\xi_1d\xi_2 \\&\lesssim \xi_3^{-2s+1+\epsilon} \int  \chi(\xi_1,-\xi_3)\chi(\xi_2,-\xi_3)d\xi_1d\xi_2  \|v_1\|_{L^2}\|v_2\|_{L^2} \lesssim \|v_1\|_{L^2}\|v_2\|_{L^2}.
\end{align*}
Thus the Schur test is valid and $I_{\sigma, \epsilon} $ is bounded.

\noindent\textit{Case b)}. Only one $\tilde{\xi}_j$ is small: w.l.o.g., $|\tilde{\xi}_1|<1$. If $|\tilde{\xi}_2|, |\tilde{\xi}_3|>|\xi|/2$, then $|\Phi|\gtrsim |\xi-\xi_1|\xi^2$ and one may follow the steps in Case 1. If $|\tilde{\xi}_2|<|\xi|/2$, then $\xi_3\sim -\xi$ and $|\tilde{\xi}_3|\gtrsim |\xi|$. Thus
\begin{align*}
I_{\sigma, \epsilon} &\lesssim \int \frac{\xi^{-2s+1+\epsilon}}{|\tilde{\xi}_1\tilde{\xi}_2\tilde{\xi}_3|^\sigma}\lesssim \int \frac{\xi^{-2s+1+\epsilon-\sigma}}{|\xi-\xi_1|^\sigma}vv_1v_2v_3 \\&\lesssim \left(\int_{|\xi-\xi_1|<1}\frac{\jap{\xi}^{-2s+1+\epsilon-\sigma}}{|\xi-\xi_1|^\sigma} v_1(\xi_1)v(\xi)d\xi_1d\xi \right)\|v_2\|_2\|v_3\|_2
\end{align*}
which is bounded since $2s+\sigma>1+\epsilon$.

\noindent\textit{Case c).} No $\tilde{\xi}_j$ is small. Then at least one $\tilde{\xi}_j$ satisfies $|\tilde{\xi}_j|\gtrsim |\xi|$: indeed, if, for example,
$$
|\tilde{\xi}_1|, |\tilde{\xi}_2|< \frac{|\xi|}{4},
$$
then $|\tilde{\xi}_3| = |2\xi-\tilde{\xi}_1-\tilde{\xi}_2| \gtrsim |\xi|$. Moreover, it is not difficult to see that, if only one $\tilde{\xi}_j$ satisfies $|\tilde{\xi}_j|\gtrsim |\xi|$, then $|\xi_3|\gtrsim|\xi|$. This means that, among $\tilde{\xi}_1,\tilde{\xi}_2,\tilde{\xi}_3$ and $\xi_3$, at least two of them are of size $\xi$ (or larger). Hence
\begin{align*}
I_{\sigma, \epsilon} &\lesssim \int_{|\xi-\xi_1|>1} \frac{\xi^{-2s+1+\epsilon-\sigma}}{|\xi-\xi_1|^{\sigma}}vv_1v_2v_3 \lesssim \int_{|\xi-\xi_1|>1} \frac{\xi^{-2s+1+\epsilon-\sigma}}{|\xi-\xi_1|^{\sigma}}vv_1d\xi d\xi_1 \left\|\int v_2v_3 d\xi_2 \right\|_{L^\infty_{\xi,\xi_1}} \\&\lesssim \left\|\int_{|\xi-\xi_1|>1} \frac{\xi^{-2s+1+\epsilon-\sigma}}{|\xi-\xi_1|^{\sigma}}vd\xi \right\|_{L^2_{\xi_1}}\|v_1\|_{L^2}\|v_2\|_{L^2}\|v_3\|_{L^2} \lesssim \left\|\frac{1}{|\cdot|^{\sigma}} \right\|_{L^{1/\sigma}_w}\left\|\jap{\xi}^{-2s+1+\epsilon-\sigma} \right\|_{L^{1/(1-\sigma)}}\prod_{j=1}^4\|v_j\|_{L^2}
\end{align*}
and the bound follows for $\sigma$ sufficiently close to 1. 
\end{proof}

\subsection{Nonlinear smoothing for the cubic (NLS)}
In this case,
$$
\Phi(\Xi)=(\xi-\xi_1)(\xi-\xi_3),\quad m(\Xi)=1.
$$
We define $\tilde{\xi}_1=\xi-\xi_1$, $\tilde{\xi}_3=\xi-\xi_3$. Due to the lack of symmetry in the phase, $\xi_2$ does not play the same role as $\xi_1,\xi_3$.

\begin{lem}
	For any $s\ge 0$, the estimate \eqref{eq:frequencyrestricted} holds for the \eqref{nls}.
\end{lem}
\begin{proof}
	Observe that 
	$$
	\frac{\partial \Phi}{\partial \xi_1}=\tilde{\xi}_3,\quad 	\frac{\partial \Phi}{\partial \xi_3}=\tilde{\xi}_1.
	$$

	Since
	$$
	\frac{\jap{\x}^s}{\jap{\x_1}^s\jap{\x_2}^s\jap{\x_3}^s} \lesssim  1,
	$$
	it suffices to prove
	$$
	\left|\int_{\substack{\x_1+\dots + \x_k=\x\\|\Phi(\Xi)-\alpha|<M}} v_1v_2v_3vd\x_1d\x_3d\x \right| \lesssim \prod_{j=1}^4 \|v_j\|_2.
	$$
	
	\noindent\textbf{Case 1.} Either $|\tilde{\xi}_1|<1$ or $|\tilde{\xi}_3|<1$. If $|\tilde{\xi}_1|<1$, the Cauchy-Schwarz inequality implies directly that
	\begin{align*}
	\left|\int_{\substack{\x_1+\dots + \x_k=\x\\|\Phi(\Xi)-\alpha|<M}} v_1v_2v_3vd\x_1d\x_3d\x \right|\lesssim \int_{|\tilde{\xi}_1|<1} \left(\int v_2v_3 d\x_3\right)\left(\int vv_1 d\x\right)d\tilde{\x}_1 \lesssim \prod_{j=1}^4 \|v_j\|_{L^2}.
	\end{align*}
	
	\noindent\textbf{Case 2.} $|\tilde{\xi}_3|, |\tilde{\xi}_1|>1$. In this case, we have
	$$
	1\le |\xi-\xi_3|< \frac{|\alpha|+ M}{|\xi-\xi_1|}< |\alpha| + M
	$$
	and one may perform the change of variables $\xi_1 \to \Phi$:
	\begin{align}
	& \left|\int_{\substack{\x_1+\dots + \x_k=\x}} \mathbbm{1}_{|\Phi-\alpha|<M} v_1v_2v_3vd\x_1d\x_3d\x \right| \lesssim  	\sup_\xi \left(\int\mathbbm{1}_{|\Phi-\alpha|<M}  d\x_1 d\x_3 \right)^{1/2}\prod_{j=1}^4\|v_j\|_2\\\lesssim &\sup_\xi \left(\int_{1<|\xi-\xi_3|<|\alpha|+ M}\frac{\mathbbm{1}_{|\Phi-\alpha|<M} }{|\xi-\xi_3|} d\Phi d\x_3 \right)^{1/2}\prod_{j=1}^4\|v_j\|_2 \lesssim \jap{M}^{1/2^+}\jap{\alpha}^{0^+}\prod_{j=1}^4\|v_j\|_2.
	\end{align}
\end{proof}

\begin{lem}
	Under the conditions of Corollary \ref{teo:nls}, estimate \eqref{eq:basicaepsilon} holds.
\end{lem}
\begin{proof}
\textbf{Case 1.} If $|\xi_3|\ll |\xi|$, then
$$
\frac{\jap{\xi}^s}{\jap{\xi_1}^s\jap{\xi_2}^s}\lesssim 1,\quad |\Phi|\gtrsim |\xi \tilde{\xi}_1|.
$$
We then estimate
\begin{align*}
I_{\sigma, \epsilon} &\lesssim \int \frac{\xi^{\epsilon-\sigma}}{|\xi-\xi_1|^{\sigma}}vv_1v_2v_3 \lesssim \int \frac{\xi^{\epsilon-\sigma}}{|\xi-\xi_1|^{\sigma}}vv_1d\xi d\xi_1 \left\|\int v_2v_3 d\xi_2 \right\|_{L^\infty_{\xi,\xi_1}} \\&\lesssim \left\|\int \frac{\xi^{\epsilon-\sigma}}{|\xi-\xi_1|^{\sigma}}vd\xi \right\|_{L^2_{\xi_1}}\|v_1\|_{L^2}\|v_2\|_{L^2}\|v_3\|_{L^2} \lesssim \left\|\frac{1}{|\cdot|^{\sigma}} \right\|_{L^{1/\sigma}_w}\left\|\jap{\xi}^{\epsilon-\sigma} \right\|_{L^{1/(1-\sigma)}}\prod_{j=1}^4\|v_j\|_{L^2}
\end{align*}
which holds as long as $\sigma>(1+\epsilon)/2$. An analogous estimate can be made if $|\xi_1|\ll |\xi|$.

\noindent\textbf{Case 2.} If $|\xi_1|, |\xi_3|\gtrsim |\xi|$, we split in two subcases:

\noindent\textit{Case a).} If $|\xi_2|\ll |\xi|$, then $|\tilde{\xi}_3|\gtrsim |\xi|$. Hence
$$
\frac{\jap{\xi}^s}{\jap{\xi_1}^s\jap{\xi_3}^s}\lesssim \jap{\xi}^{-s},\quad |\Phi|=|(\xi-\xi_1)(\xi_1-\xi_2)|\gtrsim|\xi \tilde{\xi}_1|
$$ 
 and one may proceed as in Case 1.
 
\noindent\textit{Case b).} If $|\xi_2|\gtrsim |\xi|$ and $|\tilde{\xi}_3|>1$, 
$$
\frac{\jap{\xi}^{s+\epsilon}}{\jap{\xi_1}^s\jap{\xi_2}^s\jap{\xi_3}^s\jap{\Phi}^\sigma}\lesssim \frac{\jap{\xi}^{\epsilon-2s}}{|\tilde{\xi}_1|^\sigma}.
$$
A similar estimate to that of Case 1 yields
$$
I_{\sigma, \epsilon} \lesssim \left\|\frac{1}{|\cdot|^{\sigma}} \right\|_{L^{1/\sigma}_w}\left\|\jap{\xi}^{\epsilon-2s} \right\|_{L^{1/(1-\sigma)}}\prod_{j=1}^4\|v_j\|_{L^2}.
$$
The right hand side is bounded as long as we choose $\sigma$ so that $2s+\sigma>1+\epsilon$.

\noindent\textit{Case c).}  If $|\xi_2|\gtrsim |\xi|$ and $|\tilde{\xi}_1|,|\tilde{\xi}_3|<1$, 
$$
\frac{\jap{\xi}^{s+\epsilon}}{\jap{\xi_1}^s\jap{\xi_2}^s\jap{\xi_3}^s\jap{\Phi}^\sigma}\lesssim \jap{\xi}^{\epsilon-2s} \lesssim 1
$$
and one may estimate directly
\begin{align*}
I_{\sigma, \epsilon} \lesssim \int_{|\tilde{\xi}_1|, |\tilde{\xi}_3|<1} vv_1v_2v_3d\xi d\xi_1d\xi_3 \lesssim  \int_{|\tilde{\xi}_1|, |\tilde{\xi}_3|<1} v_1v_3d\xi_1d\xi_3 \left\|\int vv_2d\xi \right\|_{L^\infty_{\xi_1,\xi_3}} \lesssim \prod_{j=1}^4 \|v_j\|_2.
\end{align*}

\end{proof}

\subsection{Nonlinear Smoothing for the (KdV)}

In the context of the (KdV), we have
$$
\Phi(\Xi)=3\xi\xi_1\xi_2,\quad m(\Xi)=\xi.
$$
Unfortunately, the $H^s$ norm will not be enough to prove estimate \eqref{eq:basicaepsilon}. From now on, assume the conditions of Corollary \ref{teo:kdv}. Choose $1<\rho,\lambda<\infty$ and $\sigma\in (1/2,1)$ such that
\begin{equation}\label{eq:defirholambda}
1-\sigma =\frac{1}{\rho} + \frac{1}{\lambda},\quad \rho(2\sigma-1-\epsilon)>1,\quad \lambda\le 2/(1-s)^+.
\end{equation}
This is possible since $\epsilon<\min\{1,s\}$. 

Define
$$
\|u\|_{\mathcal{H}^s}:= \|u\|_{H^s}  + \|\hat{u}\|_{L^\lambda(-1,1)}.
$$
Observe that, since we consider the (KdV) flow on $H^s\cap L^2(|x|^{s/2}dx)$,  $\tilde{u}\in H^{s/2}\hookrightarrow L^\lambda$. Once we prove the adapted bounds \eqref{eq:basicaepsilon} and \eqref{eq:frequencyrestricted} on $\mathcal{H}^s$, Corollary \ref{teo:kdv} follows from the analogous induction on the normal form reduction.

\begin{lem}[Adapted frequency-restricted estimate]
	Under the hypothesis of Corollary \ref{teo:kdv}, for any $u_1,u_2\in \mathcal{H}^s$,  one has
	\begin{equation}
	\|T^{\alpha,M}(u_1,u_2)\|_{\mathcal{H}^s} \lesssim \jap{M}^{1/2}\|u_1\|_{\mathcal{H}^s}\|u_2\|_{\mathcal{H}^s}.
	\end{equation}
\end{lem}
\begin{proof}
	Notice that
	$$
	\frac{\partial \Phi}{\partial \xi_1}=\xi(\xi_2-\xi_1),\quad \frac{\jap{\xi}^{s}}{\jap{\xi_1}^s\jap{\xi_2}^s}\lesssim 1.
	$$
	The proof of the $L^\lambda$ bound is trivial:
	$$
	\left\|\int_{\xi_1+\xi_2=\xi} \xi v_1(\xi_1)v_2(\x_2)d\x_1 \right\|_{L^\lambda(-1,1)}\lesssim  \left\|\int v_1(\xi_1)v_2(\x_2)d\x_1\right\|_{L^{\infty}(-1,1)} \lesssim \|v_1\|_{L^2}\|v_2\|_{L^2}.
	$$
	\textbf{Case 1.} $|\xi_2-\xi_1|\gtrsim |\xi|$. Then
	\begin{align*}
	I^{\alpha, M}=\left|\int_{\substack{\x_1+\xi_2=\x\\|\Phi(\Xi)-\alpha|<M}} \xi\frac{\jap{\xi}^{s}}{\jap{\xi_1}^s\jap{\xi_2}^s} v_1v_2vdx_1dx \right| &\lesssim \sup_\xi\left(\int \mathbbm{1}_{|\Phi-\alpha|<M}\xi^2  d\xi_1\right)^{1/2}\prod_{j=1}^3\|v_j\|_{L^2} \\&\lesssim \sup_\xi\left(\int \mathbbm{1}_{|\Phi-\alpha|<M}\frac{|\xi|}{|\xi_2-\xi_1|} d\Phi\right)^{1/2}\prod_{j=1}^3\|v_j\|_{L^2}\\& \lesssim M^{1/2}\prod_{j=1}^3\|v_j\|_{L^2}.
	\end{align*}

\textbf{Case 2.} $|\xi_2-\xi_1|\ll |\xi|$. Then
$$
\xi_1,\xi_2\sim \frac{\xi}{2},\quad \Phi \sim \frac{3}{4}\xi^3
$$	
and
\begin{align*}
I^{\alpha, M}=&\left|\int_{\substack{\x_1+\xi_2=\x\\|\Phi(\Xi)-\alpha|<M}} \xi\frac{\jap{\xi}^{s}}{\jap{\xi_1}^s\jap{\xi_2}^s} v_1v_2vdx_1dx \right| \lesssim  \sup_{\xi_1} \left(\int \xi^2 \mathbbm{1}_{|\Phi-\alpha|<M} d\xi\right)^{1/2}\prod_{j=1}^3\|v_j\|_{L^2} \\\lesssim& \left(\int \xi^2\mathbbm{1}_{|3\xi^3/4-\alpha|<M}d\xi\right)^{1/2}\prod_{j=1}^3\|v_j\|_{L^2} \lesssim M^{1/2}\prod_{j=1}^3\|v_j\|_{L^2}.
\end{align*}
\end{proof}

\begin{lem}[Adapted phase weighted estimate]\label{lem:estimativakdv}
Under the assumptions of Corollary 2, for $u_1,u_2\in \mathcal{H}^s$,  one has
	\begin{equation}
	\|T_\sigma(u_1,u_2)\|_{\mathcal{H}^{s+\epsilon}}\lesssim \|u_1\|_{\mathcal{H}^s}\|u_2\|_{\mathcal{H}^s}.
	\end{equation}
\end{lem}

\begin{proof}[Proof of Lemma \ref{lem:estimativakdv}]
Without loss of generality, assume $|\xi_1|\ge |\xi_2|$. Define
$$
\frac{1}{r}=\frac{1}{2}+\frac{1}{\rho}, \quad \frac{1}{q}=\sigma + \frac{1}{\lambda}.
$$
Then, if $|\xi_2|<1$, the Hardy-Littlewood-Sobolev inequality implies
\begin{align*}
I_{\sigma,\epsilon}&\lesssim \int \frac{\xi^{1+\epsilon-2\sigma}}{|\xi_2|^\sigma}vv_1v_2 \lesssim \|v_1\|_{L^2}\left\|\int \frac{\xi^{1+\epsilon-2\sigma}}{|\xi-\xi_1|^\sigma}v(\xi)v_2(\xi-\xi_1)d\xi \right\|_{L^2} \\&\lesssim \|v_1\|_{L^2}\left\|\frac{v}{\jap{\xi}^{2\sigma-1-\epsilon}}\right\|_{L^{r}}\left\|\frac{v_2}{|\cdot|^\sigma}\right\|_{L^{q}_w(-1,1)}\lesssim \|v_1\|_{L^2}\|v\|_{L^2} \left\|\frac{1}{\jap{\xi}^{2\sigma-1-\epsilon}} \right\|_{L^{\rho}}\left\|\frac{v_2}{|\cdot|^\sigma}\right\|_{L^{q}_w(-1,1)}\\&\lesssim  \|v_1\|_{L^2}\|v\|_{L^2}\|v_2\|_{L^\lambda}\lesssim \prod_{j=1}^3\|u_j\|_{\mathcal{H}^s}.
\end{align*}
On the other hand, if $|\xi_2|>1$,
\begin{align*}
I_{\sigma,\epsilon}\lesssim \int \frac{\xi^{1+\epsilon-2\sigma}}{\jap{\xi_2}^\sigma}vv_1v_2\lesssim \int \frac{1}{\jap{\xi_2}^\sigma}v_2 d\xi_2 \left\|\int vv_1 d\xi_1 \right\|_{L^{\infty}_{\xi_2}} \lesssim \left\|\frac{1}{\jap{\cdot}^\sigma}\right\|_{L^2}\prod_{j=1}^3\|v_j\|_{L^2}.
\end{align*}

It remains to prove the boundedness in $L^\lambda(-1,1)$:
\begin{align*}
\left\|T_\sigma(v_1,v_2)\right\|_{L^\lambda(-1,1)}\lesssim \left\|\int v_1v_2 \right\|_{L^\lambda(-1,1)}\lesssim \left\|\int v_1v_2 \right\|_{L^\infty(-1,1)}\lesssim \|v_1\|_{L^2}\|v_2\|_{L^2}.
\end{align*}
\end{proof}

\subsection{Nonlinear Smoothing for the (mZK)}
%In this section, using Theorem \ref{thm:nonlinearsmooth}, we prove that, for any $s>3/2$, a nonlinear smoothing of order $\epsilon< \min\{2s-3, 1\}$ is verified. In order to obtain the range claimed in Corollary \ref{teo:mZK}, one then applies Proposition \ref{prop:interpol}. 
%\vskip10pt
We write $\xi_j=(x_j,y_j)$. Define $\tilde{x}_j=x-x_j$, $\tilde{y}_j=y-y_j$. We have
$$
\Phi=\tilde{x}_1\tilde{x}_2\tilde{x}_3 + \tilde{y}_1\tilde{y}_2\tilde{y}_3
$$

\begin{lem}
	For any $s> 3/2$, the estimate \eqref{eq:frequencyrestricted} holds for the \eqref{mzk}.
\end{lem}

\begin{proof}
%	Observe that
%	$$
%	\frac{\partial \Phi}{\partial x_1}=\tilde{x}_2\tilde{x}_3,\quad \frac{\partial \Phi}{\partial x_2}=\tilde{x}_1\tilde{x}_3, \quad\frac{\partial \Phi}{\partial x_3}=\tilde{x}_1\tilde{x}_2.
%	$$
	We assume that
	$$
	|x|\ge |y|,\quad |\xi_1|\ge |\xi_2|\ge |\xi_3|.
	$$
	In this proof, we view the $y$-variables essentially as parameters and perform a similar analysis to that of \eqref{mkdv}. We recall that
	$$
	\frac{\partial \Phi}{\partial x_1}=3(x-x_2)(x_3-x_1), \frac{\partial \Phi}{\partial x_2}=3(x-x_1)(x_3-x_2).
	$$
	
	\textbf{Case 1.} $|\partial_{x_1}\Phi|\gtrsim |x|^2$ or $|\partial_{x_2}\Phi|\gtrsim |x|^2$. Assuming the first possibility, one may perform the change of variables $x_1\to \Phi$:
	\begin{align*}
	I^{\alpha,M}=&\int_{\substack{\x_1+\dots + \x_3=\x\\|\Phi(\Xi)-\alpha|<M}}\frac{\jap{x}^{1+s}}{\jap{\xi_1}^s\jap{\xi_2}^s\jap{\xi_3}^s }v_1v_2v_3v d\x_1d\xi_2 d\x \\\lesssim &\sup_{\xi} \left(\int \mathbbm{1}_{|\Phi-\alpha|<M}\frac{x^2}{\jap{\xi_2}^{2s}\jap{\xi_3}^{2s}}dx_1dx_3dy_2dy_3\right)^{1/2}\prod_{j=1}^4\|v_j\|_{L^2}\\ \lesssim& \sup_\xi \left(\int \mathbbm{1}_{|\Phi-\alpha|<M} \frac{1}{\jap{y_2}^{2s}\jap{x_3}^s\jap{y_3}^s}d\Phi dx_3dy_2dy_3\right)^{1/2}\prod_{j=1}^4\|v_j\|_{L^2}\\ \lesssim& M^{1/2}\prod_{j=1}^4\|v_j\|_{L^2}.
	\end{align*}
	\textbf{Case 2.} $|\partial_{x_1}\Phi|, |\partial_{x_2}\Phi|\ll |x|^2$. In this case,
	$$
	|x|\sim |x_1|\sim |x_2|\sim |x_3|, \quad \frac{\jap{x}^{1+s}}{\jap{\xi_1}^s\jap{\xi_2}^s\jap{\xi_3}^s}\lesssim \frac{1}{\jap{\xi_1}^{s-1/2}\jap{\xi_2}^{s-1/2}}
	$$
	and the restriction on $\Phi$ is not necessary:
	\begin{align*}
	I^{\alpha,M}=&\int_{\substack{\x_1+\dots + \x_3=\x\\|\Phi(\Xi)-\alpha|<M}}\frac{\jap{x}^{1+s}}{\jap{\xi_1}^s\jap{\xi_2}^s\jap{\xi_3}^s }v_1v_2v_3v d\x_1d\xi_2 d\x \\\lesssim &\sup_{\xi} \left(\int \frac{1}{\jap{\xi_1}^{2s-1}\jap{\xi_2}^{2s-1}} d\xi_1 d\xi_2\right)^{1/2}\prod_{j=1}^4\|v_j\|_{L^2} \lesssim \prod_{j=1}^4\|v_j\|_{L^2}.
	\end{align*}
	
\end{proof}

\begin{lem}
	Under the conditions of Corollary \ref{teo:mZK}, estimate \eqref{eq:basicaepsilon} holds.
\end{lem}
\begin{proof}
Once again, we make some simplifications:
$$
|x|\ge |y|,\quad |\xi_1|\ge |\xi_2|\ge |\xi_3|.
$$
\textbf{Case 1.} $|\xi_2|\gtrsim |x|$. Then
\begin{align*}
I_{\sigma,\epsilon}&\lesssim \int \frac{x^{1+\epsilon+s}}{\jap{\xi_1}^s\jap{\x_2}^s\jap{\x_3}^s }vv_1v_2v_3\lesssim \int \frac{1}{\jap{\xi_2}^{s-(1+\epsilon)/2}\jap{\xi_3}^{s-(1+\epsilon)/2} }vv_1v_2v_3\\&\lesssim \int \frac{1}{\jap{\xi_2}^{s-(1+\epsilon)/2}\jap{\xi_3}^{s-(1+\epsilon)/2} }v_2v_3d\xi_2d\xi_3\left\|\int vv_1d\xi_1\right\|_{L^\infty_{\xi_2,\xi_3}}\\&\lesssim \left\|\frac{1}{\jap{\cdot}^{s-(1+\epsilon)/2}}\right\|_{L^2}^2\prod_{j=1}^4 \|v_j\|_{L^2}.
\end{align*}
\textbf{Case 2.} $|\xi_2|\ll |x|$. In this region, we have

$$
|\tilde{y}_1|\ll |x|,\quad  |\tilde{y}_2|,|\tilde{y}_3|\lesssim |x|, \quad \tilde{x}_2\sim x,\ \tilde{x}_3\sim x.
$$

\textit{Case a).} $|\Phi|\gtrsim |x|^2$. Then, analogously to Case 1,
\begin{align*}
I_{\sigma,\epsilon}&\lesssim \int \frac{x^{1+\epsilon+s}}{\jap{\xi_1}^s\jap{\x_2}^s\jap{\x_3}^s\jap{\Phi}^\sigma }vv_1v_2v_3\lesssim \int \frac{x^{1+\epsilon-2\sigma}}{\jap{\xi_2}^{s}\jap{\xi_3}^{s} }vv_1v_2v_3\lesssim \left\|\frac{1}{\jap{\cdot}^{s}}\right\|_{L^2}^2\prod_{j=1}^4 \|v_j\|_{L^2}.
\end{align*}

\textit{Case b).} $|\Phi|\ll |x|^2$. In what follows, we recall that one may freely interchange the variables of integration $\xi\leftrightarrow \xi_1\leftrightarrow \xi_2 \leftrightarrow \xi_3$. In order to estimate the integral
$$
\int \frac{\jap{x}^{1+\epsilon}}{\jap{\xi_2}^s\jap{\xi_3}^s\jap{\Phi}^\sigma} vv_1v_2v_3 = \int \left(\int\frac{\jap{x}^{1+\epsilon}}{\jap{\xi_2}^s\jap{\xi_3}^s\jap{\Phi}^\sigma} vv_1 \right) v_2v_3d\xi_2d\xi_3 = \int K(\xi_2,\xi_3)v_2v_3d\xi_2d\xi_3,
$$
we apply Schur's test to $K$. For $\xi_2$ fixed, we want to see if the following integral is bounded:
$$
\int\frac{\jap{x}^{1+\epsilon}}{\jap{\xi_2}^s\jap{\Phi}^\sigma} vv_1d\xi_1d\xi \lesssim \int \frac{1}{\jap{y_2}^s}\left(\int \frac{x^{1+\epsilon}}{\jap{\Phi}^\sigma}vv_1dxdx_1\right)dydy_1=\int \frac{1}{\jap{y_2}^s}\left(\int K_2(x,x_1)vv_1dxdx_1\right)dydy_1.
$$
For $y,y_1$ fixed, we apply once again Schur's test to $K_2$: for $x$ fixed,
$$
\int_{|\Phi|<|x|^2} \frac{x^{1+\epsilon}}{\jap{\Phi}^\sigma} dx_1 \lesssim \int_{|\Psi|<1} \frac{1}{|\Psi|^\sigma}dx_1,\quad \Psi=x-x_1 + \frac{\tilde{y}_1\tilde{y}_2\tilde{y}_3}{\tilde{x}_2\tilde{x}_3}.
$$
Observe that
$$
\left|\frac{\partial (\Psi + x_1)}{\partial x_1}\right| =\left|\frac{\tilde{y}_1\tilde{y}_2\tilde{y}_3}{\tilde{x}_2^2\tilde{x}_3}\right| \ll 1
$$
and so one may perform the change of variables $x_1\to \Psi$:
$$
\int_{|\Psi|<1} \frac{1}{|\Psi|^\sigma}dx_1\lesssim\int_{|\Psi|<1} \frac{1}{|\Psi|^\sigma}d\Psi \lesssim 1.
$$
Since $x$ and $x_1$ play similar roles in the definition of $\Psi$, we conclude that the validity of the Schur test for $K_2$. Hence
\begin{align*}
\int \frac{1}{\jap{y_2}^s}\left(\int K_2(x,x_1)vv_1dxdx_1\right)dydy_1&\lesssim \int \frac{1}{\jap{y-y_1-y_3}^s}\|v(\cdot,y)\|_{L^2}\|v_1(\cdot,y_1)\|_{L^2}dydy_1\\&\lesssim \left\|\frac{1}{\jap{\cdot-y_3}^s}\right\|_{L^1}\|v_1\|_{L^2}\|v\|_{L^2}\lesssim \|v_1\|_{L^2}\|v\|_{L^2}.
\end{align*}
This proves that
$$
\int K(x_2,x_3)dx_2
$$
is uniformly bounded in $x_3$. Due to the symmetry of $K$, we now conclude the validity of Schur's test for $K$ and so
$$
I_{\sigma,\epsilon}\lesssim \prod_{j=1}^4\|v_j\|_{L^2}.
$$
\end{proof}

\subsection{Nonlinear Smoothing for (dNLS*)}

The case of the derivative nonlinear Schrödinger equation is slightly more complex than the previous examples for two reasons:
\begin{enumerate}
	\item We have to replace \eqref{eq:frequencyrestricted} with \eqref{eq:frequencyrestrictedgen}. As discussed in Remark \ref{gammabeta}, this poses no problem in the application of the INFR;
	\item The nonlinearity is the sum of two pure-power nonlinearities. At first, one would have to modify the INFR to this framework, taking into consideration that the normal form equation would now contain an infinite number of terms with mixed nonlinearities. Even though it should be feasible, we do not pursue this here. Instead, we follow the argument of \cite{mosincatyoon}: observe that, since $s>1/2$, $H^s(\real)$ is an algebra and so the $|u|^4u$ is bounded in $H^s$.  In the INFR procedure, after the integration by parts in time, one must replace $\tilde{u}_t$ by the two nonlinear terms. Consider the one with the quintic power. The boundedness of $|u|^4u$ in $H^s(\real)$ implies that the bounds \eqref{eq:basicaepsilon} and \eqref{eq:frequencyrestrictedgen} are directly applicable and no further integration by parts in time is necessary. By contrast, for the cubic nonlinearity, we lack a bound for $|u|^2\partial_x\bar{u}$ in $H^s$ and the INFR method must be iterated. Therefore, after the first step (where we gain the extra regularity - this does not come freely from the algebra property), the quintic term becomes irrelevant and we may handle the INFR development as if only the cubic power was present.
\end{enumerate}

Let us begin by analyzing the cubic term with loss of derivative. We have
$$
\Phi=\xi^2-\xi_1^2 + \xi_2^2 - \xi_3^2=2(\xi-\xi_1)(\xi-\xi_3)=2(\xi_2-\xi_1)(\xi_2-\xi_3), \quad m(\Xi)=\xi_2
$$

\begin{lem}
	For $s>1/2$ and $\epsilon<\min\{1,2s-1\}$, estimate \eqref{eq:basicaepsilon} holds for the cubic nonlinearity of \eqref{gdnls}.
\end{lem}
\begin{proof}

\textbf{Case 1.} $\min\{|\xi_2-\xi_1|,|\xi_2-\xi_3| \}\le 1$. Assume that $|\xi_2-\xi_1|\le1$. Then $\jap{\xi}\sim \jap{\xi_3}$ and so
$$
\frac{\xi_2\jap{\xi}^{s+\epsilon}}{\jap{\xi_1}^s\jap{\xi_2}^s\jap{\xi_3}^s}\lesssim \frac{\jap{\xi}^\epsilon}{\jap{\xi_1}^{2s-1}\jap{\Phi}^\sigma}.
$$
If $|\xi_1|\gtrsim|\xi|$, then $m\lesssim 1$ and, setting $\eta=\xi_2-\xi_1=\xi-\xi_3$,
\begin{align*}
I_{\sigma,\epsilon}\lesssim \int_{|\eta|<1}\left(\int v_1(\xi_2-\eta)v_2(\xi_2)d\xi_2\right)\left(\int v(\xi)v_3(\xi-\eta)d\xi\right)d\eta\lesssim \prod_{j=1}^4\|v_j\|_{L^2}.
\end{align*}
If $|\xi_1|\ll |\xi|$, then $\xi\sim \xi_3$ and $\Phi\sim \xi(\xi-\xi_3)$. Thus
\begin{align*}
I_{\sigma,\epsilon} \lesssim \int_{|\xi-\xi_3|<1} \frac{\xi^{\epsilon-\sigma}}{|\xi-\xi_3|^\sigma}vv_1v_2v_3d\xi_1d\xi_3d\xi \lesssim \left(\int_{|\xi-\xi_3|<1} \frac{1}{|\xi-\xi_3|^\sigma}vv_3d\xi_3d\xi\right)\|v_1\|_{L^2}\|v_2\|_{L^2} \lesssim \prod_{j=1}^4\|v_j\|_{L^2}.
\end{align*}
\vskip10pt
Thus we may now suppose from now on that $|\xi_2-\xi_1|, |\xi_2-\xi_3|>1$. Observe that this implies $\jap{\Phi}\gtrsim \jap{\xi_2-\xi_1}\jap{\xi_2-\xi_3}$.

\noindent\textbf{Case 2.} $|\xi_2|^2\lesssim \jap{\Phi}$. 
\textit{Case a).} $|\xi_2|\gtrsim |\xi|$. Then 
$$
\frac{\xi_2\jap{\xi}^{s+\epsilon}}{\jap{\xi_1}^s\jap{\xi_2}^s\jap{\xi_3}^s\jap{\Phi}^\sigma} \lesssim\frac{\xi_2\jap{\xi}^{\epsilon}}{\jap{\xi_1}^s\jap{\xi_3}^s\jap{\xi_2}^{2\sigma}} \lesssim \frac{1}{\jap{\xi_1}^s\jap{\xi_3}^s}
$$
and
\begin{align*}
I_{\sigma,\epsilon}\lesssim \sup_{\xi } \left(\int \frac{1}{\jap{\xi_1}^{2s }\jap{\xi_3}^{2s}}d\xi_1d\xi_3\right)^{1/2} \prod_{j=1}^4\|v_j\|_{L^2}.
\end{align*} 

\textit{Case b).} $|\xi_2|\ll |\xi|$. Assume that $|\xi_1|\gtrsim  |\xi|$. Then
$$
\frac{\xi_2\jap{\xi}^{s+\epsilon}}{\jap{\xi_1}^s\jap{\xi_2}^s\jap{\xi_3}^s} \lesssim \frac{\jap{\xi}^\epsilon}{\jap{\xi_2}^s\jap{\xi_3}^s\jap{\Phi}^{\sigma-1/2}}.
$$
If $|\xi_3|\gtrsim |\xi|$,
\begin{align*}
I_{\sigma,\epsilon}\lesssim \sup_{\xi} \left(\int \frac{1}{\jap{\xi_2}^{2s}\jap{\xi_3}^{2(s-\epsilon)}\jap{\xi_2-\xi_3}^{2\sigma-1}}d\xi_2 d\xi_3\right)^{1/2}\prod_{j=1}^4\|v_j\|_{L^2} \lesssim\prod_{j=1}^4\|v_j\|_{L^2}.
\end{align*}
If $|\xi_3|\ll |\xi|$, then
$$
\frac{\partial \Phi}{\partial \xi_1} = \xi-\xi_3 \sim \xi
$$
and
\begin{align*}
I_{\sigma,\epsilon}\lesssim \sup_\xi \left(\int \frac{\jap{\xi}^{2\epsilon} \xi_2^2}{\jap{\xi_2}^{2s}\jap{\Phi}^{2\sigma} }d\xi_1d\xi_2\right)^{1/2}\prod_{j=1}^4\|v_j\|_{L^2} \lesssim \sup_{\xi} \left(\int \frac{\jap{\xi}^{2\epsilon-1}}{\jap{\xi_2}^{2s-2}\jap{\Phi}^{2\sigma}}d\Phi d\xi_2\right)^{1/2}\prod_{j=1}^4\|v_j\|_{L^2}.
\end{align*}
Since $\jap{\Phi}\gtrsim \jap{\xi}, |\xi_2|^2$, if one chooses
$$
a>\max\{3/2-s, 0\}, b=2\epsilon-1,
$$
\begin{align*}
I_{\sigma,\epsilon}\lesssim \sup_{\xi} \left(\int \frac{1}{\jap{\xi_2}^{2s-2+2a}\jap{\Phi}^{2\sigma-a-b}}d\Phi d\xi_2\right)^{1/2}\prod_{j=1}^4\|v_j\|_{L^2} \lesssim\prod_{j=1}^4\|v_j\|_{L^2}.
\end{align*}

\textbf{Case 3.} $|\xi_2|^2\gg \jap{\Phi}$. Since $\jap{\Phi}\gtrsim \jap{\xi_2-\xi_1}\jap{\xi_2-\xi_3}$, we may suppose that $|\xi_2-\xi_1|\ll |\xi_2|$. Hence $|\xi_1|\sim |\xi_2|$.

\textit{Case a).} If $|\xi_1|\gg |\xi_3|$, then $|\xi|\ll |\xi_2|$ (otherwise $|\Phi|\gtrsim |\xi-\xi_3||\xi_2-\xi_3|\gtrsim|\xi_2|^2$). Either $|\xi_3|\gtrsim |\xi|$ and
\begin{align}\label{eq:dnls3a}
I_{\sigma, \epsilon}&\lesssim \sup_\xi \left(\int \frac{|\xi_2|^2\jap{\xi}^{2\epsilon} }{ \jap{\xi_1}^{2s}\jap{\xi_2}^{2s}\jap{\xi_2-\xi_1}^{2\sigma}\jap{\xi_2-\xi_3}^{2\sigma} }d\x_1 d\xi_2\right)^{1/2}\prod_{j=1}^4\|v_j\|_{L^2}\nonumber\\&\lesssim \sup_\xi \left(\int \frac{1}{ \jap{\xi_1}^{2s}\jap{\xi_2}^{2s + 2\sigma-2\epsilon-2}\jap{\xi_2-\xi_1}^{2\sigma} }d\x_1 d\xi_2\right)^{1/2}\prod_{j=1}^4\|v_j\|_{L^2}
\end{align}
or $|\xi_3|\ll |\xi|$ and
\begin{align*}
I_{\sigma, \epsilon}&\lesssim \sup_\xi \left(\int \frac{|\xi_2|^2\jap{\xi}^{2\epsilon} }{ \jap{\xi_2}^{2s}\jap{\xi_3}^{2s}\jap{\xi-\xi_3}^{2\sigma}\jap{\xi_2-\xi_3}^{2\sigma} }d\x_2 d\xi_3\right)^{1/2}\prod_{j=1}^4\|v_j\|_{L^2}\\&\lesssim \sup_\xi \left(\int \frac{1}{ \jap{\xi_3}^{2s}\jap{\xi_2}^{2s + 2\sigma-2}\jap{\xi_2-\xi_3}^{2\sigma} }d\x_1 d\xi_2\right)^{1/2}\prod_{j=1}^4\|v_j\|_{L^2}.
\end{align*}
\textit{Case b).} $|\xi_3|\gtrsim |\xi_1|$. If $|\xi_3|\gg |\xi_1|\sim |\xi_2|$, one proceeds as in \eqref{eq:dnls3a}. Otherwise, one has $|\xi_1|\sim |\xi_2|\sim |\xi_3|\sim |\xi|$ and
\begin{align*}
I_{\sigma, \epsilon}&\lesssim \sup_\xi \left(\int \frac{|\xi_2|^2\jap{\xi}^{2s+2\epsilon} }{ \jap{\xi_1}^{2s}\jap{\xi_2}^{2s}\jap{\xi_3}^{2s}\jap{\xi-\xi_3}^{2\sigma}\jap{\xi-\xi_1}^{2\sigma}}d\x_1 d\xi_3\right)^{1/2}\prod_{j=1}^4\|v_j\|_{L^2}\\&\lesssim \sup_\xi \left(\int \frac{1}{ \jap{\xi_1}^{2s-1-\epsilon}\jap{\xi_3}^{2s-1-\epsilon} \jap{\xi-\xi_3}^{2\sigma}\jap{\xi-\xi_1}^{2\sigma}}d\x_1 d\xi_3\right)^{1/2}\prod_{j=1}^4\|v_j\|_{L^2}.
\end{align*}

\end{proof}

\begin{lem}
	For $s>1/2$, estimate \eqref{eq:frequencyrestrictedgen} holds for the cubic nonlinearity of \eqref{gdnls}.
\end{lem}
\begin{proof}
	
If $|\xi_1|\gtrsim |\xi|$ and $|\xi_3|\gtrsim |\xi_2|$, then
$$
\frac{\xi_2 \jap{\xi}^s}{\jap{\xi_1}^s\jap{\x_2}^s\jap{\x_3}^s}\lesssim 1
$$
and the estimate follows from Lemma 6. On the other hand, if $|\xi_1|\gtrsim |\xi|$ and $|\xi_2|\gg |\xi_3|$, 
$$
\frac{\partial \Phi}{\partial \xi_1} = \frac{\partial ((\xi_2-\xi_1)(\xi_2-\xi_3))}{\partial \xi_1}\sim \xi_2
$$
and
\begin{align*}
I^{\alpha, M} &\lesssim \sup_\xi \left(\int \frac{\xi_2^2}{\jap{\xi_2}^{2s}\jap{\xi_3}^{2s}  }\mathbbm{1}_{|\Phi-\alpha|<M} d\xi_1d\xi_3\right)^{1/2}\prod_{j=1}^4\|v_j\|_{L^2} \\&\lesssim \sup_\xi \left(\int \frac{1}{\jap{\xi_3}^{2s}  }\mathbbm{1}_{|\Phi-\alpha|<M} d\Phi d\xi_3\right)^{1/2}\lesssim M^{1/2}\prod_{j=1}^4\|v_j\|_{L^2}.
\end{align*}

By symmetry in $\xi_1, \xi_3$, the only remaining case is when $|\xi_1|, |\xi_3|\ll |\xi|$, which means that $\xi_2\sim \xi$. In this situation, since
$$
\frac{\partial \Phi}{\partial \xi_3} \sim \xi, \Phi\sim \xi^2, %\quad \left| \xi_1 - \left(\xi - \frac{\alpha}{\tilde{\xi}_3}\right)\right|<\frac{M}{\tilde{\x}_3}\lesssim \frac{M}{\xi}.
$$
one has
\begin{align*}
I^{\alpha, M}&\lesssim \sup_\xi \left(\int \frac{\xi_2^2}{\jap{\xi_1}^{2s}\jap{\xi_3}^{2s}}\mathbbm{1}_{|\Phi-\alpha|<M}d\xi_1d\xi_3\right)^{1/2}\prod_{j=1}^4\|v_j\|_{L^2}\\&\lesssim \sup_\xi \left(\int \frac{|\Phi|^{1/2}}{\jap{\xi_1}^{2s}}\mathbbm{1}_{|\Phi-\alpha|<M}d\xi_1d\Phi\right)^{1/2}\prod_{j=1}^4\|v_j\|_{L^2} \\&\lesssim \sup_\xi \left(\int |\Phi|^{1/2} \mathbbm{1}_{|\Phi-\alpha|<M}d\Phi\right)^{1/2}\prod_{j=1}^4\|v_j\|_{L^2} \lesssim \sup\{\jap{\alpha}^{1/4},\jap{M}^{1/4}\}\jap{M}^{1/2}\prod_{j=1}^4\|v_j\|_{L^2}.
\end{align*}

\end{proof}

Now consider the quintic term in \eqref{gdnls}. Since $|u|^4u$ is bounded in $H^s$, the estimate \eqref{eq:frequencyrestrictedgen} is straightforward.

\begin{lem}
	Under the assumptions of Corollary \ref{cor:dnls}, the estimate \eqref{eq:basicaepsilon} holds for the quintic term of \eqref{gdnls}.
\end{lem}

\begin{proof}
	
Here,
$$
\Phi = \xi^2- \xi_1^2 + \xi_2^2 - \xi_3^2 + \xi_4^2 - \xi_5^2.
$$
The lack of a nice algebraic factorization of $\Phi$ forces us to apply more robust arguments, which do not use the specific structure of the zero set of $\Phi$. Consequently, the total smoothing effect will be reduced to $\epsilon \le 1/2$.

We order the frequencies $|\xi_1|\ge |\xi_2| \ge \dots \ge |\xi_5|$. In particular, $|\xi_1|\gtrsim |\xi|$. We use the estimate
$$
I_{\sigma,\epsilon} \lesssim \sup_\xi \left( \int \frac{\jap{\xi}^{2\epsilon}}{\jap{\xi_2}^{2s}\dots \jap{\xi_5}^{2s}\jap{\Phi}^{2\sigma}  }d\xi_1 \dots d\xi_4 \right)^{1/2}\prod_{j=1}^6\|v_j\|_{L^2}.
$$
\textbf{Case 1.} $|\xi_2|\gtrsim |\xi|$. Either $|\xi_5|\gtrsim |\xi_2|$ and
\begin{align*}
I_{\sigma,\epsilon} \lesssim \sup_\xi \left( \int \frac{1}{\jap{\xi_2}^{2s-\epsilon}\jap{\xi_3}^{2s}\jap{\xi_4}^{2s} \jap{\xi_5}^{2s-\epsilon}  }d\xi_2 \dots d\xi_5 \right)^{1/2}\prod_{j=1}^6\|v_j\|_{L^2}
\end{align*}
or
$$
\left|\frac{\partial \Phi}{\partial \xi_2}\right| = 2|\xi_2-\xi_5| \gtrsim |\xi|
$$
and, using the change of variables $\xi_2 \mapsto \Phi$,
\begin{align*}
I_{\sigma,\epsilon} \lesssim \sup_\xi \left( \int \frac{\jap{\xi}^{2\epsilon-1}}{\jap{\xi_3}^{2s}\jap{\xi_4}^{2s} \jap{\xi_5}^{2s} \jap{\Phi}^{2\sigma }} d\Phi d\xi_3 \dots d\xi_5 \right)^{1/2}\prod_{j=1}^6\|v_j\|_{L^2}.
\end{align*}
\textbf{Case 2.} $|\xi_2|\ll |\xi|$. Since
$$
\left|\frac{\partial \Phi}{\partial \xi_1}\right| = 2|\xi_1-\xi_5| \gtrsim |\xi|,
$$
the change of variables $\xi_1\mapsto \Phi$ yields
\begin{align*}
I_{\sigma,\epsilon} \lesssim \sup_\xi \left( \int \frac{\jap{\xi}^{2\epsilon-1}}{\jap{\xi_2}^{2s}\jap{\xi_3}^{2s}\jap{\xi_4}^{2s}  \jap{\Phi}^{2\sigma }} d\Phi d\xi_2 \dots d\xi_4 \right)^{1/2}\prod_{j=1}^6\|v_j\|_{L^2}.
\end{align*}

\end{proof}

\section{Acknowledgements}
The authors wish to thank Felipe Linares for fruitful discussions and comments.
S. Correia was partially supported by Funda\c{c}\~ao para a Ci\^encia e Tecnologia, through the grant UID/MAT/04561/2019. J. Drumond Silva was partially supported by Funda\c{c}\~ao para a Ci\^encia e Tecnologia, through the grant UID/MAT/04459/2019.

\appendix
\section{Nonlinear interpolation and persistence properties for dispersive equations}

\begin{prop}[\cite{tartar}]\label{prop:tartar}
	Given Banach spaces $A_0\subset A_1$ and $B_0\subset B_1$ and a mapping $K:A_1\to B_1$, suppose that
$$
\|Ku-Kv\|_{B_1}\le f(\|u\|_{A_1}, \|v\|_{A_1})\|u-v\|_{A_1}
$$
and
$$
\|Ku\|_{B_0}\le g(\|u\|_{A_1})\|u\|_{A_0},
$$
for some continuous $f,g$ with values in $\real^+$. Then, for any $0<\theta<1$ and $1\le p\le +\infty$,
$$
\|Ku\|_{(B_0,B_1)_{\theta, p}} \le h(\|u\|_{A_1})\|u\|_{(A_0,A_1)_{\theta,p}}\quad \mbox{for all }u\in (A_0,A_1)_{\theta,p},
$$ 
where $h(t)=g(2t)^{1-\theta}f(t,2t)^\theta$.
\end{prop}

Given $s,r>0$, define $\Sigma^{s,r}=H^s(\real)\cap L^2(\jap{x}^{2r}dx)$. It is well-known that
$$
(H^s(\real), L^2(\real))_{\theta, 2}= H^{s\theta}(\real),\quad (L^2(\jap{x}^{2r}dx), L^2(\real))_{\theta,2} = L^{2}(\jap{x}^{2r\theta}dx)
$$

\begin{prop}\label{prop:sigma}
	For any $0<\theta<1$, $k\in \mathbb{N}$ and $r>0$,
	$$
	(\Sigma^{k,r}, L^2(\real))_{\theta,2}=\Sigma^{k\theta,r\theta}.	
$$
\end{prop}

\begin{proof}
	We set $E=L^2(\real)$, $F=H^k(\real)$ and define $A:D(A)\to E$ by $Au=\jap{x}^ru$. It is clear that $D(A)=L^2(\jap{x}^{2r}dx)$ and that $A$ is closed.
	Given $t>0$, 
	$$t(A+t)^{-1}v=\frac{t}{\jap{x}^r+t}v
	$$ 
 is a bounded operator on both $E$ and $F$ (uniformly in $t$). The result now follows from \cite{grisvard} (see also \cite{peetre}).
\end{proof}

\begin{thm}\label{teo:persistkdv}
	For any $s\ge 2r >0$, if $u_0\in \Sigma^{s,r}$, then the corresponding local solution $u\in C([0,T], H^s(\real))$ of \eqref{kdv} satisfies $u\in L^\infty((0,T), \Sigma^{s,r})$.
\end{thm}

\begin{proof}
	The result is already known for $s\ge 1$ (\cite{nahasponce}, \cite{nahas2}).
	For the sake of simplicity, we only prove the case $s<1$ and $s=2r$.
	Fix $R>0$. Define $T=T(R)$ such that, for any $u_0\in L^2(\real)$ with $\|u_0\|_{L^2}<R$, the corresponding solution $u$ of \eqref{kdv} is defined on $[0, T]$ and satisfies
	$$
	\|u\|_{C([0,T], L^2(\real^d))}\lesssim \|u_0\|_{L^2}.
	$$
	Define
	$$
	A_0= B_0=\Sigma^{2,1},\quad   A_1=B_1=L^2(\real).
	$$
	For any $0<t<T$, consider the mapping
	$$
	Ku_0 = \left\{\begin{array}{ll}
	(R-\|u_0\|_{L^2})u(t), & \|u_0\|_{L^2}<R\\
	0, & \|u_0\|_{L^2}\ge R
	\end{array}\right..
	$$
	From \cite{kato}, it follows that $K:A_0\to B_0$, and
	$$
	\|Ku_0\|_{B_0}\lesssim C(\|u_0\|_{A_1})\|u_0\|_{A_0}, \quad u_0\in A_0.
	$$
	Moreover, from the local well-posedness theory (see, for example, \cite{KPV5}), 
	$$
	\|Ku_0-Kv_0\|_{B_1}\lesssim C(\|u_0\|_{A_1}, \|v_0\|_{A_1})\|u_0-v_0\|_{A_1}, \quad u_0,v_0 \in A_1.
	$$
	Hence, by Proposition \ref{prop:tartar} and Proposition \ref{prop:sigma}, 
	$$
	\|u(t)\|_{\Sigma^{s,s/2}}= \|u(t)\|_{(\Sigma^{2,1}, L^2(\real))_{s,2}} \lesssim_R \|u_0\|_{(\Sigma^{2,1}, L^2(\real))_{s,2}} = \|u_0\|_{\Sigma^{s,s/2}}.
	$$
	Since $R$ can be chosen arbitrarily large, the proof is finished. 
\end{proof}

\bibliography{biblio}
\bibliographystyle{plain}

\begin{center}
	{\scshape Simão Correia}\\
	{\footnotesize
		Centro de Matemática, Aplicações Fundamentais e Investigação Operacional,\\
		Department of Mathematics,\\
		Instituto Superior T\'ecnico, Universidade de Lisboa\\
		Av. Rovisco Pais, 1049-001 Lisboa, Portugal\\
		simao.f.correia@tecnico.ulisboa.pt
	}
\pagebreak

	{\scshape Jorge Drummond Silva}\\
{\footnotesize
	Center for Mathematical Analysis, Geometry and Dynamical Systems,\\
	Department of Mathematics,\\
	Instituto Superior T\'ecnico, Universidade de Lisboa\\
	Av. Rovisco Pais, 1049-001 Lisboa, Portugal\\
	jsilva@math.tecnico.ulisboa.pt
}

\end{center}

\end{document}